\documentclass[reqno]{amsart}

\usepackage{amsmath}
\usepackage{amssymb}
\usepackage{amsfonts}
\usepackage{amsthm}

\usepackage[utf8]{inputenc}
\usepackage[T1]{fontenc}

\usepackage{dsfont}

\usepackage[mathcal]{euscript}

\usepackage[dvipsnames,svgnames]{xcolor}
\colorlet{MyBlue}{DodgerBlue!75!Black}

\usepackage{subfigure}
\usepackage{tikz}

\usepackage{acronym}
\usepackage{paralist}
\usepackage{wasysym}
\usepackage{xspace}

\usepackage[sort&compress]{natbib}

\usepackage{hyperref}
\hypersetup{
colorlinks=true,
linktocpage=true,
pdfstartview=FitH,
breaklinks=true,
pdfpagemode=UseNone,
pageanchor=true,
pdfpagemode=UseOutlines,
plainpages=false,
bookmarksnumbered,
bookmarksopen=false,
bookmarksopenlevel=1,
hypertexnames=true,
pdfhighlight=/O,
urlcolor=Maroon,linkcolor=MyBlue!70!Black,citecolor=DarkGreen!70!Black, 
pdftitle={},
pdfauthor={},
pdfsubject={},
pdfkeywords={},
pdfcreator={pdfLaTeX},
pdfproducer={LaTeX with hyperref}
}

\newcommand{\R}{\mathbb{R}}
\newcommand{\Q}{\mathbb{Q}}

\DeclareMathOperator{\argdot}{\cdotp}

\DeclareMathOperator{\bd}{bd}

\DeclareMathOperator{\ex}{\mathbb{E}}

\DeclareMathOperator{\one}{\mathds{1}}

\DeclareMathOperator{\prob}{\mathbb{P}}

\DeclareMathOperator{\supp}{supp}

\newcommand{\bvec}{e}
\newcommand{\dd}{\:d}

\newcommand{\eps}{\varepsilon}
\newcommand{\exclude}[1]{\operatorname\setminus\left\{#1\right\}}
\newcommand{\filter}{\mathcal{F}}
\newcommand{\from}{\colon}

\newcommand{\pd}{\partial}
\DeclareMathOperator{\simplex}{\Delta}
\newcommand{\intsimplex}{\simplex^{\circ}}
\newcommand{\wilde}{\widetilde}

\newcommand{\abs}[1]{\left\lvert #1 \right\rvert}

\newcommand{\norm}[1]{\left\| #1 \right\|}

\newcommand{\braket}[2]{\left\langle #1 \middle\vert  #2 \right\rangle}
\newcommand{\smallbraket}[2]{\langle #1 \vert  #2 \rangle}

\newcommand{\dis}{\displaystyle}
\newcommand{\txs}{\textstyle}

\newcommand{\insum}{\sum\nolimits}

\usepackage[textwidth=2.1cm]{todonotes}
\usepackage{soul}
\setstcolor{red}
\sethlcolor{SkyBlue}

\theoremstyle{plain}
\newtheorem{theorem}{Theorem}

\newtheorem*{corollary*}{Corollary}
\newtheorem{lemma}[theorem]{Lemma}
\newtheorem{proposition}[theorem]{Proposition}

\theoremstyle{definition}
\newtheorem{definition}[theorem]{Definition}
\newtheorem*{definition*}{Definition}

\theoremstyle{remark}
\newtheorem{remark}{Remark}
\newtheorem*{remark*}{Remark}

\numberwithin{equation}{section}
\numberwithin{theorem}{section}
\numberwithin{remark}{section}
\numberwithin{example}{section}

\newcommand{\play}{\mathcal{N}}
\newcommand{\act}{\mathcal{A}}
\newcommand{\set}{\mathcal{S}}
\newcommand{\pay}{u}
\newcommand{\payv}{v}
\newcommand{\modpayv}{\payv^{\sigma}}
\newcommand{\mutpayv}{\payv^{\eta}}
\newcommand{\strat}{\mathcal{X}}
\newcommand{\intstrat}{\strat^{\circ}}
\newcommand{\game}{\mathfrak{G}}
\newcommand{\modgame}{\game^{\sigma}}
\newcommand{\mutgame}{\game^{\eta}}

\newcommand{\dkl}{D_{\textup{KL}}}

\newcommand{\eq}{x^{\ast}}

\begin{document}


\title
{Imitation dynamics with payoff shocks}

\author[P.~Mertikopoulos]{Panayotis Mertikopoulos}
\address
[P.~Mertikopoulos]
{CNRS (French National Center for Scientific Research), LIG, F-38000 Grenoble, France\\
and
Univ. Grenoble Alpes, LIG, F-38000 Grenoble, France}
\email{\href{mailto:panayotis.mertikopoulos@imag.fr}{panayotis.mertikopoulos@imag.fr}}
\urladdr{\url{http://mescal.imag.fr/membres/panayotis.mertikopoulos}}

\author[Y.~Viossat]{Yannick Viossat}
\address
[Y.~Viossat]
{PSL, Université Paris\textendash Dauphine, CEREMADE UMR7534, Place du Maréchal de Lattre de Tassigny, 75775 Paris, France}
\email{\href{mailto:viossat@ceremade.dauphine.fr}{viossat@ceremade.dauphine.fr}}

\thanks{Supported in part by the French National Research Agency under grant no. GAGA--13--JS01--0004--01 and the French National Center for Scientific Research (CNRS) under grant no. PEPS--GATHERING--2014.}


\begin{abstract}

We investigate the impact of payoff shocks on the evolution of large populations of myopic players that employ simple strategy revision protocols such as the ``imitation of success''.
In the noiseless case, this process is governed by the standard (deterministic) replicator dynamics;
in the presence of noise however, the induced stochastic dynamics are different from previous versions of the stochastic replicator dynamics (such as the aggregate-shocks model of \citealp{FH92}).
In this context, we show that strict equilibria are always stochastically asymptotically stable, irrespective of the magnitude of the shocks;
on the other hand, in the high-noise regime, non-equilibrium states may also become stochastically asymptotically stable and dominated strategies may survive in perpetuity (they become extinct if the noise is low).
Such behavior is eliminated if players are less myopic and revise their strategies based on their cumulative payoffs.
In this case, we obtain a second order stochastic dynamical system whose attracting states coincide with the game's strict equilibria and where dominated strategies become extinct (a.s.), no matter the noise level.
\keywords{%
Dominated strategies
\and
evolutionary dynamics
\and
replicator dynamics
\and
revision protocols
\and
aggregate payoff shocks
\and
strict equilibria.}
\end{abstract}

\maketitle


\newacro{iid}[i.i.d.]{independent and identically distributed}
\newacro{SDE}{stochastic differential equation}
\newacro{KL}{Kullback\textendash Leibler}


\section{Introduction}
\label{sec:introduction}

Evolutionary game dynamics study the evolution of behavior in populations of boundedly rational agents that interact strategically.
The most widely studied dynamical model in this context is the replicator dynamics:
introduced in biology as a model of natural selection \citep{TJ78}, the replicator dynamics also arise from models of imitation of successful individuals \citep{Sch98,Wei95,BW96} and from models of learning in games \citep{Rus99,HSV09,MM10}.
Mathematically, they stipulate that the growth rate of the frequency of a strategy is proportional to the difference between the payoff of individuals playing this strategy and the mean payoff in the population.
These payoffs are usually assumed deterministic:
this is typically motivated by a large population assumption and the premise that, owing to the law of large numbers, the resulting mean field provides a good approximation of a more realistic but less tractable stochastic model. 
This approach makes sense when the stochasticity affecting payoffs is independent across individuals playing the same strategies, but it fails when the payoff shocks are aggregate, that is, when they affect all individuals playing a given strategy in a similar way.

Such aggregate shocks are not uncommon.
\cite{Ber14} recounts the story of squirrels stocking nuts for the winter months:
squirrels may stock a few or a lot of nuts, the latter leading to a higher probability of surviving a long winter but a higher exposure to predation.
The unpredictable mildness or harshness of the ensuing winter will then favor one of these strategies in an aggregate way (see also \citealp[Sec.~3.1.1]{RS11}, and references therein).
In traffic engineering, one might think of a choice of itinerary to go to work:
fluctuations of traffic on some roads affect all those who chose them in a similar way. 
Likewise, in data networks, a major challenge occurs when trying to minimize network latencies in the presence of stochastic disturbances:
in this setting, the travel time of a packet in the network does not depend only on the load of each link it traverses, but also on unpredictable factors such as random packet drops and retransmissions, fluctuations in link quality, excessive backlog queues, etc. \citep{BG92}.

Incorporating such aggregate payoff shocks in the biological derivation of the replicator dynamics leads to the stochastic replicator dynamics of \cite{FH92}, later studied by (among others) \cite{Cab00}, \cite{Imh05} and \cite{HI09}.
To study the long-run behavior of these dynamics, \cite{Imh05} introduced a modified game where the expected payoff of a strategy is penalized by a term which increases with the variance of the noise affecting this strategy's payoff (see also \citealp{HI09}).
Among other results, it was then shown that
\begin{inparaenum}
[\itshape a\upshape)]
\item
strategies that are iteratively (strictly) dominated in this modified game become extinct almost surely;
and
\item
strict equilibria of the modified game are stochastically asymptotically stable.
\end{inparaenum}

In this biological model, noise is detrimental to the long-term survival of strategies:
a strategy which is strictly dominant on average (i.e. in the original, unmodified game) but which is affected by shocks of substantially higher intensity becomes extinct almost surely.
By contrast, in the learning derivation of the replicator dynamics, noise leads to a stochastic exponential learning model where only iteratively undominated strategies survive, irrespective of the intensity of the noise \citep{MM10};
as a result the frequency of a strictly dominant strategy converges to $1$ almost surely.
Moreover, strict Nash equilibria of the original game remain stochastically asymptotically stable (again, independently of the level of the noise), so the impact of the noise in the stochastic replicator dynamics of exponential learning is minimal when compared to the stochastic replicator dynamics with aggregate shocks.

In this paper, we study the effect of payoff shocks when the replicator equation is seen as a model of imitation of successful agents.
As in the case of \cite{Imh05} and \cite{HI09}, it is convenient to introduce a noise-adjusted game which is reduced to the original game in the noiseless, deterministic regime. 
We show that:
\begin{inparaenum}
[\itshape a\upshape)]
\item
strategies that are iteratively strictly dominated in the modified game become extinct almost surely;
and
\item
strict equilibria of the modified game are stochastically asymptotically stable.
\end{inparaenum}
However, despite the formal similarity, our results are qualitatively different from those of \cite{Imh05} and \cite{HI09}:
in the modified game induced by imitation of success in the presence of noise, noise is not detrimental per se.
In fact, in the absence of differences in expected payoffs, a strategy survives with a probability that does not depend on the random variance of its payoffs:
a strategy's survival probability is simply its initial frequency.
Similarly, even if a strategy which is strictly dominant in expectation is subject to arbitrarily high noise, it will always survive with positive probability;
by contrast, such strategies become extinct (a.s.) in the aggregate shocks model of \cite{FH92}. 

That said, the dynamics' long-term properties change dramatically if players are less ``myopic'' and, instead of imitating strategies based on their instantaneous payoffs, they base their decisions on the cumulative payoffs of their strategies over time.
In this case, we obtain a second-order stochastic replicator equation which can be seen as a noisy version of the higher order game dynamics of \cite{LM13}.
Thanks to this payoff aggregation mechanism, the noise averages out in the long run and we recover results that are similar to those of \cite{MM10}:
strategies that are dominated in the original game become extinct (a.s.) and strict Nash equilibria attract nearby initial conditions with arbitrarily high probability.

\subsection{Paper Outline}
\label{sec:outline}

The remainder of our paper is structured as follows:
in Section \ref{sec:model}, we present our model and we derive the stochastic replicator dynamics induced by imitation of success in the presence of noise.
Our long-term rationality analysis begins in Section \ref{sec:analysis} where we introduce the noise-adjusted game discussed above and we state our elimination and stability results in terms of this modified game.
In Section \ref{sec:cumulative}, we consider the case where players imitate strategies based on their cumulative payoffs and we show that the adjustment due to noise is no longer relevant.
Finally, in Section \ref{sec:discussion}, we discuss some variants of our core model related to different noise processes.

\subsection{Notational conventions}
\label{sec:notation}

The real space spanned by a finite set $\set = \{s_{\alpha}\}_{\alpha=1}^{d+1}$ will be denoted by $\R^{\set}$ and we will write $\{\bvec_{s}\}_{s\in\set}$ for its canonical basis;
in a slight abuse of notation, we will also use $\alpha$ to refer interchangeably to either $s_{\alpha}$ or $\bvec_{\alpha}$ and we will write $\delta_{\alpha\beta}$ for the Kronecker delta symbols on $\set$.
The set $\simplex(\set)$ of probability measures on $\set$ will be identified with the $d$-dimensional simplex $\simplex = \{x\in \R^{\set}: \sum_{\alpha} x_{\alpha} =1 \text{ and }x_{\alpha}\geq 0\}$ of $\R^{\set}$ and the relative interior of $\simplex$ will be denoted by $\intsimplex$;
also, the support of $p\in\simplex(\set)$ will be written $\supp(p) = \{\alpha\in\set: p_{\alpha}>0\}$.
For simplicity, if $\{\set_{k}\}_{k\in\play}$ is a finite family of finite sets, we use the shorthand $(\alpha_{k};\alpha_{-k})$ for the tuple $(\dotsc,\alpha_{k-1},\alpha_{k},\alpha_{k+1},\dotsc)$ and we write $\sum_{\alpha}^{k}$ instead of $\sum_{\alpha\in\set_{k}}$. 
Unless mentioned otherwise, deterministic processes will be represented by lowercase letters, while their stochastic counterparts will be denoted by the corresponding uppercase letter.
Finally, we will suppress the dependence of the law of a process $X(t)$ on its initial condition $X(0) = x$, and we will write $\prob$ instead of $\prob_{\!x}$.

\section{The model}
\label{sec:model}

In this section, we recall a few preliminaries from the theory of population games and evolutionary dynamics, and we introduce the stochastic game dynamics under study.

\subsection{Population games}
\label{sec:games}

Our main focus will be games played by populations of nonatomic players.
Formally, such games consist of a finite set of player \emph{populations} $\play = \{1,\dotsc,N\}$ (assumed for simplicity to have unit mass), each with a finite set of \emph{pure strategies} (or \emph{types}) $\act_{k} = \{\alpha_{k,1},\alpha_{k,2},\dotsc\}$, $k\in\play$.
During play, each player chooses a strategy and the \emph{state} of each population is given by the distribution $x_{k} = (x_{k\alpha})_{\alpha\in\act_{k}}$ of players employing each strategy $\alpha\in\act_{k}$.
Accordingly, the \emph{state space} of the $k$-th population is the simplex $\strat_{k} \equiv \simplex(\act_{k})$ and the state space of the game is the product $\strat \equiv \prod_{k} \strat_{k}$.

The payoff to a player of population $k\in\play$ playing $\alpha\in\act_{k}$ is determined by the corresponding \emph{payoff function} $\payv_{k\alpha}\from \strat\to\R$ (assumed Lipschitz).%
\footnote{Note that we are considering general payoff functions and not only multilinear (resp. linear) payoffs arising from asymmetric (resp. symmetric) random matching in finite $N$-person (resp. $2$-person) games.
This distinction is important as it allows our model to cover e.g. general traffic games as in \cite{San10}.}
Thus, given a population state $x\in\strat$, the average payoff to population $k$ will be
\begin{equation}
\label{eq:pay-average}
	\insum_{\alpha}^{k} x_{k\alpha} \payv_{k\alpha}(x)
	= \braket{\payv_{k}(x)}{x},
\end{equation}
where $\payv_{k}(x) \equiv (\payv_{k\alpha}(x))_{\alpha\in\act_{k}}$ denotes the \emph{payoff vector} of the $k$-th population in the state $x\in\strat$.
Putting all this together, a \emph{population game} is then defined as a tuple $\game \equiv \game(\play,\act,\payv)$ of nonatomic player populations $k\in\play$, their pure strategies $\alpha\in\act_{k}$ and the associated payoff functions $\payv_{k\alpha}\from\strat\to\R$.

In this context, we say that a pure strategy $\alpha\in\act_{k}$ is \emph{dominated} by $\beta\in\act_{k}$ if
\begin{equation}
\label{eq:dom-pure}
\payv_{k\alpha}(x)
	< \payv_{k\beta}(x)
	\quad
	\text{for all $x\in\strat$,}
\end{equation}
i.e. the payoff of an $\alpha$-strategist is always inferior to that of a $\beta$-strategist.
More generally (and in a slight abuse of terminology), we will say that $p_{k}\in\strat_{k}$ is dominated by $p_{k}'\in\strat_{k}$ if
\begin{equation}
\label{eq:dom-mixed}
\braket{\payv_{k}(x)}{p_{k}}
	< \smallbraket{\payv_{k}(x)}{p_{k}'}
	\quad
	\text{for all $x\in\strat$,}
\end{equation}
i.e. when the average payoff of a small influx of mutants in population $k$ is always greater when they are distributed according to $p_{k}'$ rather than $p_{k}$ (irrespective of the incumbent population state $x\in\strat$).

Finally, we will say that the population state $\eq\in\strat$ is at \emph{Nash equilibrium} if
\begin{equation}
\label{eq:Nash}
\tag{NE}
\payv_{k\alpha}(\eq)
	\geq \payv_{k\beta}(\eq)
	\quad
	\text{for all $\alpha\in\supp(\eq_{k})$ and for all $\beta\in\act_{k}$, $k\in\play$.}
\end{equation}
In particular, if $\eq$ is pure (in the sense that $\supp(\eq)$ is a singleton) and \eqref{eq:Nash} holds as a strict inequality for all $\beta\notin\supp(\eq_{k})$, $\eq$ will be called a \emph{strict equilibrium}.

\begin{remark}
Throughout this paper, we will be suppressing the population index $k\in\play$ for simplicity, essentially focusing in the single-population case. 
This is done only for notational clarity:
all our results apply as stated to the multi-population model described in detail above.
\end{remark}

\subsection{Revision protocols}
\label{sec:revision}

A fundamental evolutionary model in the context of population games is provided by the notion of a \emph{revision protocol}.
Following \citet[Chapter~3]{San10}, it is assumed that each nonatomic player receives an opportunity to switch strategies at every ring of an independent Poisson alarm clock, and this decision is based on the payoffs associated to each strategy and the current population state.
The players' revision protocol is thus defined in terms of the \emph{conditional switch rates} $\rho_{\alpha\beta} \equiv \rho_{\alpha\beta}(\payv,x)$ that determine the relative mass $dx_{\alpha\beta}$ of players switching from $\alpha$ to $\beta$ over an infinitesimal time interval $dt$:%
\footnote{In other words, $\rho_{\alpha\beta}$ is the probability of an $\alpha$-strategist becoming a $\beta$-strategist up to normalization by the alarm clocks' rate.}
\begin{equation}
\label{eq:mass}
dx_{\alpha\beta}
	= x_{\alpha} \rho_{\alpha\beta} \dd t.
\end{equation}
The population shares $x_{\alpha}$ are then governed by the \emph{revision protocol dynamics}:
\begin{equation}
\label{eq:RPD}
\dot x_{\alpha}
	= \insum_{\beta} x_{\beta} \rho_{\beta\alpha}
	- x_{\alpha} \insum_{\beta} \rho_{\alpha\beta},
\end{equation}
with $\rho_{\alpha\alpha}$ defined arbitrarily.

In what follows, we will focus on revision protocols of the general ``imitative'' form
\begin{equation}
\label{eq:imitation}
\rho_{\alpha\beta}(\payv,x)
	= x_{\beta} r_{\alpha\beta}(\payv,x),
\end{equation}
corresponding to the case where a player imitates the strategy of a uniformly drawn opponent with probability proportional to the so-called \emph{conditional imitation rate} $r_{\alpha\beta}$ (assumed Lipschitz).
In particular, one of the most widely studied revision protocols of this type is the ``imitation of success'' protocol \citep{Wei95} where the imitation rate of a given target strategy is proportional to its payoff,%
\footnote{Modulo an additive constant which ensures that $\rho$ is positive but which cancels out when it comes to the dynamics.}
i.e.
\begin{equation}
\label{eq:success}
r_{\alpha\beta}(\payv,x)
	= \payv_{\beta}
\end{equation}
On account of \eqref{eq:RPD}, the mean evolutionary dynamics induced by \eqref{eq:success} take the form:
\begin{equation}
\tag{RD}
\label{eq:RD}
\dot x_{\alpha}
	= x_{\alpha} \left[ \payv_{\alpha}(x) - \insum_{\beta} x_{\beta} \payv_{\beta}(x) \right],
\end{equation}
which is simply the classical replicator equation of \cite{TJ78}.%

The replicator dynamics have attracted significant interest in the literature and their long-run behavior is relatively well understood.
For instance, \cite{Aki80}, \cite{Nac90} and \cite{SZ92} showed that dominated strategies become extinct under \eqref{eq:RD},
whereas the (multi-population) ``folk theorem'' of evolutionary game theory \citep{HS03} states that
\begin{inparaenum}
[\itshape a\upshape)]
\item
(Lyapunov) stable states are Nash;
\item
limits of interior trajectories are Nash;
and
\item
strict Nash equilibria are asymptotically stable under \eqref{eq:RD}.
\end{inparaenum}

\subsection{Payoff shocks and the induced dynamics}
\label{sec:dynamics}

Our main goal in this paper is to investigate the rationality properties of the replicator dynamics in a setting where the players' payoffs are subject to exogenous stochastic disturbances.
To model these ``payoff shocks'', we assume that the players' payoffs at time $t$ are of the form $\hat\payv_{\alpha}(t) = \payv_{\alpha}(x(t)) + \xi_{\alpha}(t)$ for some zero-mean ``white noise'' process $\xi_{\alpha}$.
Then, in Langevin notation, the replicator dynamics \eqref{eq:RD} become:
\begin{flalign}
\label{eq:SRD-Langevin}
\frac{dX_{\alpha}}{dt}
	&= X_{\alpha} \left[ \hat\payv_{\alpha} - \insum_{\beta} X_{\beta} \hat\payv_{\beta} \right]
	\notag\\
	&= X_{\alpha} \left[ \payv_{\alpha}(X) - \insum_{\beta} X_{\beta} \payv_{\beta}(X) \right]
	+ X_{\alpha} \left[ \xi_{\alpha} - \insum_{\beta} X_{\beta} \xi_{\beta} \right],
\end{flalign}
or, in \ac{SDE} form:
\begin{equation}
\label{eq:SRD}
\tag{SRD}
\begin{aligned}
dX_{\alpha}
	&= X_{\alpha} \left[ \payv_{\alpha}(X) - \insum_{\beta} X_{\beta} \payv_{\beta}(X) \right] dt
	\\
	&+ X_{\alpha} \left[ \sigma_{\alpha}(X) \dd W_{\alpha} - \insum_{\beta} X_{\beta} \sigma_{\beta}(X) \dd W_{\beta} \right],
\end{aligned}
\end{equation}
where the diffusion coefficients $\sigma_{\alpha}\from\strat\to\R$ (assumed Lipschitz) measure the intensity of the payoff shocks and the Wiener processes $W_{\alpha}$ are assumed independent.

The stochastic dynamics \eqref{eq:SRD} will constitute the main focus of this paper, so some remarks are in order:

\begin{remark}
With $\payv$ and $\sigma$ assumed Lipschitz, it follows that \eqref{eq:SRD} admits a unique (strong) solution $X(t)$ for every initial condition $X(0)\in\strat$.
Moreover, since the drift and diffusion terms of \eqref{eq:SRD} all vanish at the boundary $\bd(\strat)$ of $\strat$, standard arguments can be used to show that these solutions exist (a.s.) for all time, and that $X(t)\in\intstrat$ for all $t\geq0$ if $X(0)\in\intstrat$ \citep{Oks07,Kha12}.
\end{remark}

\begin{remark}
The independence assumption for the Wiener processes $W_{\alpha}$ can be relaxed without qualitatively affecting our analysis;%
\footnote{An important special case where it makes sense to consider correlated shocks is if the payoff functions $\payv_{\alpha}(x)$ are derived from random matchings in a finite game whose payoff matrix is subject to stochastic perturbations.
This specific disturbance model is discussed in Section \ref{sec:discussion}.}
in particular, as we shall see in the proofs of our results, 
the rationality properties of \eqref{eq:SRD} can be formulated directly in terms of the quadratic (co)variation of the noise processes $W_{\alpha}$.
Doing so however would complicate the relevant expressions considerably, so, for clarity, we will retain this independence assumption throughout our paper.
\end{remark}

\begin{remark}
The deterministic replicator dynamics \eqref{eq:RD} are also the governing dynamics for the  ``pairwise proportional imitation'' revision protocol \citep{Sch98} where a revising agent imitates the strategy of a randomly chosen opponent only if the opponent's payoff is higher than his own, and he does so with probability proportional to the payoff difference.
Formally, the conditional switch rate $\rho_{\alpha\beta}$ under this revision protocol is:
\begin{equation}
\label{eq:pairwise}
\rho_{\alpha\beta}
	= x_{\beta} \big[ \payv_{\beta} - \payv_{\alpha} \big]_{+},
\end{equation}
where $[x]_{+} = \max\{x,0\}$ denotes the positive part of $x$.
Accordingly, if the game's payoffs at time $t$ are of the perturbed form $\hat\payv_{\alpha}(t) = \payv_{\alpha}(x(t)) + \xi_{\alpha}(t)$ as before, \eqref{eq:RPD} leads to the master stochastic equation:
\begin{flalign}
\dot X_{\alpha}
	&= \insum_{\beta} X_{\beta} X_{\alpha} \big[ \hat\payv_{\alpha} - \hat\payv_{\beta} \big]_{+}
	- X_{\alpha} \insum_{\beta} X_{\beta} \big[ \hat\payv_{\beta} - \hat\payv_{\alpha} \big]_{+}
	\notag\\
	&= X_{\alpha} \insum_{\beta} X_{\beta} \left\{
	\big[ \hat\payv_{\alpha} - \hat\payv_{\beta} \big]_{+} - \big[ \hat\payv_{\beta} - \hat\payv_{\alpha} \big]_{+}
	\right\}
	= X_{\alpha} \insum_{\beta} X_{\beta} \left( \hat\payv_{\alpha} - \hat\payv_{\beta} \right)
	\notag\\
	&= X_{\alpha} \left[ \hat\payv_{\alpha} - \insum_{\beta} X_{\beta} \hat\payv_{\beta} \right],
\end{flalign}
which is simply the stochastic replicator dynamics \eqref{eq:SRD-Langevin}.
In other words, \eqref{eq:SRD} could also be interpreted as the mean dynamics of a pairwise imitation process with perturbed payoff comparisons as above.
\end{remark}

\subsection{Related stochastic models}
\label{sec:related}

The replicator dynamics were first introduced in biology, as a model of frequency-dependent selection.
They arise from the geometric population growth equation:
\begin{equation}
\label{eq:Z}
\dot{z}_{\alpha} = z_{\alpha} \payv_{\alpha} 
\end{equation} where $z_{\alpha}$ denotes the absolute population size of the $\alpha$-th genotype of a given species.%
\footnote{The replicator equation \eqref{eq:RD} is obtained simply by computing the evolution of the frequencies $x_{\alpha}= z_{\alpha}/\sum_{\beta} z_{\beta}$ under \eqref{eq:Z}.}
This biological model was also the starting point of \cite{FH92} who added aggregate payoff shocks to \eqref{eq:Z} based on the geometric Brownian model:
\begin{equation}
\label{eq:Zsto}
dZ_{\alpha}
	= Z_{\alpha} \left[ \payv_{\alpha} \dd t + \sigma_{\alpha} \dd W_{\alpha} \right],
\end{equation}
where
the diffusion process $\sigma_{\alpha} \dd W_{\alpha}$ represents the impact of random, weather-like effects on the genotype's fitness (see also \citealp{Cab00,Imh05,HI09}).%
\footnote{\cite{KP06} also considered a related evolutionary model with Stratonovich-type perturbations while, more recently, \cite{Vla12} studied the effect of discontinuous semimartingale shocks incurred by catastrophic, earthquake-like events.}
Itô's lemma applied to the population shares $X_{\alpha} = Z_{\alpha}/\insum_{\beta} Z_{\beta}$ then yields the \emph{replicator dynamics with aggregate shocks:}
\begin{equation}
\label{eq:ASRD}
\begin{aligned}
dX_{\alpha}
	&= X_{\alpha} \left[ \payv_{\alpha}(X) - \insum_{\beta} X_{\beta} \payv_{\beta}(X) \right] dt
	\\
	&+ X_{\alpha} \left[ \sigma_{\alpha} \dd W_{\alpha} - \insum_{\beta} \sigma_{\beta} X_{\beta} \dd W_{\beta} \right]
	\\
	&- X_{\alpha} \left[ \sigma_{\alpha}^{2} X_{\alpha} - \insum_{\beta} \sigma_{\beta}^{2} X_{\beta}^{2} \right] dt.
\end{aligned}
\end{equation}

In a repeated game context, the replicator dynamics also arise from a continuous-time variant of the exponential weight algorithm introduced by \cite{Vov90} and \cite{LW94} (see also \citealp{Sor09}).
In particular, if players follow the exponential learning scheme:
\begin{equation}
\begin{aligned}
dy_{\alpha}
	&= \payv_{\alpha}(x) \dd t,
	\\
x_{\alpha}
	&= \frac{\exp(y_{\alpha})}{\insum_{\beta} \exp(y_{\beta})},
\end{aligned}
\end{equation}
that is, if they play a logit best response to the vector of their cumulative payoffs, then the frequencies $x_{\alpha}$ follow \eqref{eq:RD}.\footnote{The intermediate variable $y_{\alpha}$ should be thought of as an evaluation of how good the strategy $\alpha$ is, and the formula for $x_{\alpha}$ as a way of transforming these evaluations into a strategy.}
Building on this, \cite{MM09,MM10} considered the stochastically perturbed exponential learning scheme:
\begin{equation}
\label{eq:SXL}
\begin{aligned}
dY_{\alpha}
	&= \payv_{\alpha}(X) \dd t + \sigma_{\alpha}(X) dW_{\alpha},
	\\
X_{\alpha}
	&= \frac{\exp(Y_{\alpha})}{\insum_{\beta} \exp(Y_{\beta})},
\end{aligned}
\end{equation}
where the cumulative payoffs are perturbed by the observation noise process $\sigma_{\alpha} dW_{\alpha}$.
By Itô's lemma, we then obtain the stochastic replicator dynamics of exponential learning:
\begin{equation}
\label{eq:SXRD}
\begin{aligned}
dX_{\alpha}
	&= X_{\alpha} \left[ \payv_{\alpha}(X) - \insum_{\beta} X_{\beta} \payv_{\beta}(X) \right] dt
	\\
	&+ X_{\alpha} \left[ \sigma_{\alpha} \dd W_{\alpha} - \insum_{\beta} \sigma_{\beta} X_{\beta} \dd W_{\beta} \right]
	\\
	&+\frac{X_{\alpha}}{2} \left[
	\sigma_{\alpha}^{2} (1 - 2X_{\alpha}) - \insum_{\beta} \sigma_{\beta}^{2} X_{\beta}(1 - 2X_{\beta})
	\right] dt.
\end{aligned}
\end{equation}

Besides their very distinct origins, a key difference between the stochastic replicator dynamics \eqref{eq:SRD} and the stochastic models \eqref{eq:ASRD}/\eqref{eq:SXRD} is that there is no Itô correction term in the former.
The reason for this is that in \eqref{eq:ASRD} and \eqref{eq:SXRD}, the noise affects primarily the evolution of an intermediary variable (the absolute population sizes $Z_{\alpha}$ and the players' cumulative payoffs $Y_{\alpha}$ respectively) before being carried over to the evolution of the strategy shares $X_{\alpha}$.
By contrast, the payoff shocks that impact the players' revision protocol in \eqref{eq:SRD} affect the corresponding strategy shares directly, so there is no intervening Itô correction.


\paragraph{The pure noise case.}

To better understand the differences between our model and previous models of stochastic replicator dynamics, it is useful to consider the case of pure noise, that is, when the expected payoff of each strategy is equal to one and the same constant $C$:
\(
\payv_{\alpha}(x) = C
\)
for all $\alpha\in\act$ and for all $x\in\strat$.

For simplicity, let us also assume that $\sigma_{\alpha}(x)$ is independent of the state of the population $x$.
Eq.~\eqref{eq:Zsto} then becomes a simple geometric Brownian motion of the form:
\begin{equation}
\label{eq:Zsto2}
dZ_{\alpha}
	= Z_{\alpha} \left[ C \dd t + \sigma_{\alpha} \dd W_{\alpha} \right],
\end{equation}
which readily yields $Z_{\alpha}(t) = Z_{\alpha}(0)\exp\left( (C - \sigma_{\alpha}^{2}/2) t + \sigma_{\alpha} W_{\alpha}(t) \right)$.
The corresponding frequency $X_{\alpha} = Z_{\alpha} / \insum_{\beta} Z_{\beta}$ will then be:
\begin{equation}
\label{eq:Xsto2}
X_{\alpha}(t)
	= 	\frac{X_{\alpha}(0) \exp\left(- \frac{1}{2}\sigma_{\alpha}^{2} t+ \sigma_{\alpha} W_{\alpha}(t)\right)}
	{\sum_{\beta} X_{\beta}(0) \exp\left(- \frac{1}{2}\sigma_{\beta}^{2} t+ \sigma_{\beta} W_{\beta}(t)\right)}.
\end{equation}
If $\sigma_{\alpha} \neq 0$, the law of large numbers yields $-\frac{1}{2} \sigma_{\alpha}^{2} t + \sigma_{\alpha} W_{\alpha}(t) \sim  -  \frac{1}{2} \sigma_{\alpha}^{2} t $ (a.s.).
Therefore, letting $\sigma_{\min}=\min_{\alpha \in \act} \sigma_{\alpha}$, it follows from \eqref{eq:Xsto2} that strategy $\alpha\in\act$ is eliminated if $\sigma_{\alpha}> \sigma_{\min}$ and survives if $\sigma_{\alpha}= \sigma_{\min}$ (a.s.).%
\footnote{Elimination is obvious;
for survival, simply add $\frac{1}{2}\sigma_{\min}^{2}t$ to the exponents of \eqref{eq:Xsto2} and recall that any Wiener process has $\limsup_{t} W(t) > 0$ and $\liminf_{t} W(t) <0$ (a.s.).}
In particular, if all intensities  are equal ($\sigma_{\alpha}= \sigma_{\min}$ for all $\alpha\in\act$), then all strategies survive and the share of each strategy oscillates for ever, occasionally taking values arbitrarily close to $0$ and arbitrarily close to $1$.
On the other hand, under the stochastic replicator dynamics of exponential learning for the pure noise case,
\eqref{eq:SXRD} readily yields:
\begin{equation}
X_{\alpha}(t)
	= \frac{X_{\alpha}(0)\exp(\sigma_{\alpha} W_{\alpha}(t))}{\insum_{\beta} X_{\beta}(0) \exp(\sigma_{\beta}W_{\beta}(t))}.
\end{equation}
Therefore, for any value of the diffusion coefficients $\sigma_{\alpha}$ (and, in particular, even if some strategies are affected by noise much more than others), all pure strategies survive. 

Our model behaves differently from both \eqref{eq:ASRD} and \eqref{eq:SXRD}:
in the pure noise case, for any value of the noise coefficients $\sigma_{\alpha}$ (as long as $\sigma_{\alpha}> 0$ for all $\alpha$), only a single strategy survives (a.s.), and strategy $\alpha$ survives with probability equal to $X_{\alpha}(0)$.
To see this, consider first the model with pure noise and only two strategies, $\alpha$ and $\beta$.
Then, letting $X(t) = X_{\alpha}(t)$ (so $X_{\beta}(t) = 1- X(t)$), we get:
\begin{equation}
\label{eq:pn2strat}
\dd X(t)
	= X(t) (1-X(t)) \left[ \sigma_{\alpha} \dd W_{\alpha} - \sigma_{\beta} \dd W_{\beta} \right]
	= X(t)(1-X(t))\, \sigma \dd {W}(t),
\end{equation}
where $\sigma^{2} = \sigma_{\alpha}^{2} + \sigma_{\beta}^{2}$ and we have used the time-change theorem for martingales to write $\sigma \dd W = \sigma_{\alpha} \dd W_{\alpha} - \sigma_{\beta} \dd W_{\beta}$ for some Wiener process $W(t)$.
This diffusion process can be seen as a continuous-time random walk on $[0, 1]$ with step sizes that get smaller as $X$ approaches $\{0,1\}$.
Thus, at a heuristic level, when $X(t)$ starts close to $X=1$ and takes one step to the left followed by one step to the right (or the opposite), the walk does not return to its initial position, but will approach $1$ (of course, the same phenomenon occurs near $0$).
This suggests that the process should eventually converge to one of the vertices:
indeed, letting $f(x) = \log x(1-x)$, Itô's lemma yields
\begin{equation}
df(X)
	= (1- 2X)\, \sigma \dd W - \frac{1}{2} \left[ (1-X)^{2} + X^{2} \right]\, \sigma^2 \dd t
	\leq (1- 2X) \sigma \dd W - \frac{1}{4} \sigma^{2} \dd t,
\end{equation}
so, by Lemma \ref{lem:Wbound}, we get $\lim_{t\to\infty} f(X(t)) = 0$ (a.s.), that is, $\lim_{t\to\infty} X(t) \in \{0,1\}$. 

More generally, consider the model with pure noise and $n$ strategies.
Then, computing $d[\log X_{\alpha} (1 - X_{\alpha})]$ as above, we readily obtain $\lim_{t \to \infty} X_{\alpha}(t) \in \{0, 1\}$ (a.s.), for every strategy $\alpha\in\act$ with $\sigma_{\alpha}>0$.
Since $X_{\alpha}$ is a martingale, we will have $\ex[X_{\alpha}(t)] = X_{\alpha}(0)$ for all $t\geq0$,%
\footnote{We are implicitly assuming here deterministic initial conditions, i.e. $X(0) = x$ (a.s.) for some $x\in\strat$.}
so $X_{\alpha} \to 1$ with probability $X_{\alpha}(0)$ and $X_{\alpha}(t) \to 0$ with probability $1- X_{\alpha}(0)$.%
\footnote{If several strategies are unaffected by noise, that is, are such that $\sigma_{\alpha}=0$, then their relative shares remain constant (that is, if $\alpha$ and $\beta$ are two such strategies, then $X_{\alpha}(t)/X_{\beta}(t) = X_{\alpha}(0)/X_{\beta}(0)$ for all $t\geq0$).
It follows from this observation and the above result that, almost surely, all these strategies are eliminated or all these strategies survive (and only them).}

The above highlights two important differences between our model and the stochastic replicator dynamics of \cite{FH92}.
First, in our model, noise is not detrimental in itself:
in the pure noise case, the expected frequency of a strategy remains constant, irrespective of the noise level;
by contrast, in the model of \cite{FH92}, the expected frequency of strategies affected by strong payoff noise decreases.%
\footnote{In the pure noise case of the model of \cite{FH92}, what remains constant is the expected number of individuals playing a strategy.
A crucial point here is that this number may grow to infinity.
What happens to strategies affected by large aggregate shocks is that with small probability, the total number of individuals playing this strategy gets huge, but with a large probability (going to 1), it gets small (at least compared to the number of individuals playing other strategies).
This can be seen as a gambler's ruin phenomenon, which explains that even with a higher expected payoff than others (hence a higher expected subpopulation size), the frequency of a strategy may go to zero almost surely (see e.g. \citealp[Sec~3.1.1]{RS11}). This cannot happen in our model since noise is added directly to the frequencies (which are bounded).}
Second, our model behaves in a somewhat more ``unpredictable'' way:
for instance, in the model of \cite{FH92}, when there are only two strategies with the same expected payoff, and if one of the strategies is affected by a stronger payoff noise, then it will be eliminated (a.s.);
in our model, we cannot say in advance whether it will be eliminated or not.

\section{Long-term rationality analysis}
\label{sec:analysis}

In this section, we investigate the long-run rationality properties of the stochastic dynamics \eqref{eq:SRD};
in particular, we focus on the elimination of dominated strategies and the stability of equilibrium play.

\subsection{Elimination of dominated strategies}
\label{sec:dominated}

We begin with the elimination of dominated strategies.
Formally, given a trajectory of play $x(t) \in \strat$, we say that a pure strategy $\alpha\in\act$ \emph{becomes extinct along $x(t)$} if $x_{\alpha}(t) \to 0$ as $t\to\infty$.
More generally, following \cite{SZ92}, we will say that the mixed strategy $p\in\strat$ becomes extinct along $x(t)$ if $\min\{x_{\alpha}(t): \alpha\in\supp(p)\} \to 0$ as $t\to\infty$;
otherwise, we say that $p$ survives.

Now, with a fair degree of hindsight, it will be convenient to introduce a modified game $\modgame \equiv \modgame(\play,\act,\modpayv)$ with payoff functions $\modpayv_{\alpha}$ adjusted for noise as follows:
\begin{equation}
\label{eq:pay-mod}
\modpayv_{\alpha}(x)
	= \payv_{\alpha}(x)
	- \frac{1}{2} (1 - 2x_{\alpha})\,\sigma_{\alpha}^{2}(x).
\end{equation}
\cite{Imh05} introduced a similar modified game to study the long-term convergence and stability properties of the stochastic replicator dynamics with aggregate shocks \eqref{eq:ASRD} and showed that strategies that are dominated in this modified game are eliminated (a.s.) \textendash\ cf. Remark \ref{rem:modified} below.
Our main result concerning the elimination of dominated strategies under \eqref{eq:SRD} is of a similar nature:

\begin{theorem}
\label{thm:dominated}
Let $X(t)$ be an interior solution orbit of the stochastic replicator dynamics \eqref{eq:SRD}.
Assume further that $p\in\strat$ is dominated by $p'\in\strat$ in the modified game $\modgame$.
Then, $p$ becomes extinct along $X(t)$ \textup(a.s.\textup).
\end{theorem}

\begin{remark}
As a special case, if the (pure) strategy $\alpha\in\act$ is dominated by the (pure) strategy $\beta\in\act$,
Theorem \ref{thm:dominated} shows that $\alpha$ becomes extinct under \eqref{eq:SRD} as long as
\begin{equation}
\label{eq:dom-cond-const}
\payv_{\beta}(x) - \payv_{\alpha}(x)
	> \frac{1}{2} \left[ \sigma_{\alpha}^{2}(x) + \sigma_{\beta}^{2}(x) \right]
	\quad
	\text{for all $x\in\strat$.}
\end{equation}
In terms of the original game, this condition can be interpreted as saying that $\alpha$ is dominated by $\beta$ by a margin no less that $\frac{1}{2}\max_{x} \big(\sigma_{\alpha}^{2}(x) + \sigma_{\beta}^{2}(x)\big)$.
Put differently, Theorem \ref{thm:dominated} shows that dominated strategies in the original, unmodified game become extinct provided that the payoff shocks are mild enough.
\end{remark}


\begin{proof}[Proof of Theorem \ref{thm:dominated}]
Following \cite{Cab00}, we will show that $p$ becomes extinct along $X(t)$ by studying the ``cross-entropy'' function:
\begin{equation}
\label{eq:V}
V(x)
	= \dkl(p,x) - \dkl(p',x)
	= \insum_{\alpha} \left(p_{\alpha} \log p_{\alpha} - p_{\alpha}' \log p_{\alpha}' \right)
	+ \insum_{\alpha} (p_{\alpha}' - p_{\alpha}) \log x_{\alpha},
\end{equation}
where $\dkl(p,x) = \insum_{\alpha} p_{\alpha} \log(p_{\alpha}/x_{\alpha})$ denotes the \ac{KL} divergence of $x$ with respect to $p$.
By a standard argument \citep{Wei95}, $p$ becomes extinct along $X(t)$ if $\lim_{t\to\infty} \dkl(p,X(t)) = \infty$;
thus, with $\dkl(p',x) \geq 0$, it suffices to show that $\lim_{t\to\infty} V(X(t)) = \infty$.

To that end, let  $Y_{\alpha} = \log X_{\alpha}$ so that
\begin{equation}
\label{eq:dY1}
dY_{\alpha}
	= \frac{dX_{\alpha}}{X_{\alpha}}
	- \frac{1}{2} \frac{1}{X_{\alpha}^{2}} \left(dX_{\alpha}\right)^{2},
\end{equation}
by Itô's lemma.
Then, writing
$dS_{\alpha} = X_{\alpha} \big[ \sigma_{\alpha} \dd W_{\alpha} - \insum_{\beta} X_{\beta} \sigma_{\beta} \dd W_{\beta} \big]$ for the martingale term of \eqref{eq:SRD}, we readily obtain:
\begin{flalign}
\label{eq:dT}
\left(dS_{\alpha}\right)^{2}
	&= X_{\alpha}^{2}\,
	\left[ \sigma_{\alpha} \dd W_{\alpha} - \insum_{\beta} X_{\beta} \sigma_{\beta} \dd W_{\beta} \right]
	\cdot \left[ \sigma_{\alpha} \dd W_{\alpha} - \insum_{\gamma} X_{\gamma} \sigma_{\gamma} \dd W_{\gamma} \right]
	\notag\\
	&= X_{\alpha}^{2}\,
	\left[ (1 - 2 X_{\alpha}) \sigma_{\alpha}^{2} + \insum_{\beta} \sigma_{\beta}^{2} X_{\beta}^{2}\right] dt,
\end{flalign}
where we have used the orthogonality conditions $dW_{\beta}\cdot dW_{\gamma} = \delta_{\beta\gamma} \dd t$.
By the same token, we also get $(dX_{\alpha})^{2} = (dS_{\alpha})^{2}$, and hence:
\begin{equation}
\begin{aligned}
\label{eq:dY2}
dY_{\alpha}
	&= \big( \payv_{\alpha} - \braket{\payv}{X} \big) \dd t
	- \frac{1}{2}\left[ (1 - 2 X_{\alpha}) \sigma_{\alpha}^{2} + \insum_{\beta} \sigma_{\beta}^{2} X_{\beta}^{2}\right] dt
	\\
	&+ \sigma_{\alpha} \dd W_{\alpha} - \insum_{\beta} X_{\beta} \sigma_{\beta} \dd W_{\beta}.
\end{aligned}
\end{equation}
Therefore, after some easy algebra, we obtain:
\begin{flalign}
\label{eq:dV1}
dV
	&= \insum_{\alpha} \left( p_{\alpha}' - p_{\alpha} \right) dY_{\alpha}
	\notag\\
	&= \smallbraket{\payv(X)}{p' - p} \dd t
	\notag\\
	&- \frac{1}{2} \insum_{\alpha} \big( p_{\alpha}' - p_{\alpha} \big)
	(1 - 2 X_{\alpha})\,\sigma_{\alpha}^{2}(X) \dd t
	+ \insum_{\alpha} \left( p_{\alpha}' - p_{\alpha} \right) \sigma_{\alpha}(X) \dd W_{\alpha}
	\notag\\
	&= \braket{\modpayv(X)}{p' - p} \dd t
	+ \insum_{\alpha} (p_{\alpha}' - p_{\alpha}) \sigma_{\alpha}(X) \dd W_{\alpha}
\end{flalign}
where we have used the fact that $\insum_{\alpha} \left( p_{\alpha}' - p_{\alpha} \right) = 0$.

Now, since $p$ is dominated by $p'$ in $\modgame$, we will have $\braket{\modpayv(x)}{p' - p} \geq m$ for some positive constant $m>0$ and for all $x\in\strat$.
Eq.~\eqref{eq:dV1} then yields:
\begin{equation}
\label{eq:dV2}
V(X(t))
	\geq V(X(0)) + mt + \xi(t),
\end{equation}
where $\xi$ denotes the martingale part of \eqref{eq:dV1}, viz.
\begin{equation}
\label{eq:xi}
\xi(t)
	= \insum_{\alpha} \left( p_{\alpha}' - p_{\alpha} \right) \int_{0}^{t}
	\sigma_{\alpha}(X(s)) \dd W_{\alpha}(s).
\end{equation}
Since the $\sigma(X(t))$ is bounded and continuous (a.s.), Lemma \ref{lem:Wbound} shows that $mt + \xi(t) \sim mt$ as $t\to\infty$, so the RHS of \eqref{eq:dV2} escapes to $\infty$ as $t\to\infty$. 
This implies $\lim_{t\to\infty} V(X(t)) = \infty$ and our proof is complete.
\end{proof}

Theorem \ref{thm:dominated} is our main result concerning the extinction of dominated strategies under \eqref{eq:SRD} so a few remarks are in order:

\begin{remark}
Theorem \ref{thm:dominated} is analogous to the elimination results of \citet[Theorem 3.1]{Imh05} and \citet[Prop.~1A]{Cab00} who show that dominated strategies become extinct under the replicator dynamics with aggregate shocks \eqref{eq:ASRD} if the shocks satisfy certain ``tameness'' requirements.
On the other hand, Theorem \ref{thm:dominated} should be contrasted to the corresponding results of \cite{MM10} who showed that dominated strategies become extinct under the stochastic replicator dynamics of exponential learning \eqref{eq:SXRD} \emph{irrespective} of the noise level (for a related elimination result, see also \citealp{BM14}).
The crucial qualitative difference here lies in the Itô correction term that appears in the drift of the stochastic replicator dynamics:
the Itô correction in \eqref{eq:SXRD} is ``just right'' with respect to the logarithmic variables $Y_{\alpha} = \log X_{\alpha}$ and this is what leads to the unconditional elimination of dominated strategies.
On the other hand, even though there is no additional drift term in \eqref{eq:SRD} except for the one driven by the game's payoffs, the logarithmic transformation $Y_{\alpha} = \log X_{\alpha}$ incurs an Itô correction which is reflected in the definition of the modified payoff functions \eqref{eq:pay-mod}.
\end{remark}

\begin{remark}
A standard induction argument based on the rounds of elimination of iteratively dominated strategies (see e.g. \citealp{Cab00} or \citealp{MM10}) can be used to show that the only strategies that survive under the stochastic replicator dynamics \eqref{eq:SRD} must be iteratively undominated in the modified game $\modgame$.
\end{remark}

\begin{remark}
Finally, it is worth mentioning that \cite{Imh05} also establishes an exponential rate of extinction of dominated strategies under the stochastic replicator dynamics with aggregate shocks \eqref{eq:ASRD}.
Specifically, if $\alpha\in\act$ is dominated, \cite{Imh05} showed that there exist constants $A,B>0$ and $A',B'>0$ such that
\begin{equation}
\label{eq:dom-decay}
X_{\alpha}(t)
	= o\left(\exp\left(-At + B\sqrt{t \log \log t}\right)\right)
	\quad
	\text{\textup(a.s.\textup)},
\end{equation}
and
\begin{equation}
\label{eq:dom-largedev}
\prob\left[ X_{\alpha}(t) > \eps \right]
	\leq \frac{1}{2} \mathrm{erfc}\left[ A' t^{1/2} + B' \log\eps \cdot t^{-1/2}\right] ,
\end{equation}
provided that the noise coefficients of \eqref{eq:ASRD} satisfy a certain ``tameness'' condition.
Following the same reasoning, it is possible to establish similar exponential decay rates for the elimination of dominated strategies under \eqref{eq:SRD}, but the exact expressions for the constants in \eqref{eq:dom-decay} and \eqref{eq:dom-largedev} are more complicated, so we do not present them here.
\end{remark}

\subsection{Stability analysis of equilibrium play}
\label{sec:stability}

In this section, our goal will be to investigate the stability and convergence properties of the stochastic replicator dynamics \eqref{eq:SRD} with respect to equilibrium play.
Motivated by a collection of stability results that is sometimes called the ``folk theorem'' of evolutionary game theory \citep{HS03}, we will focus on the following three properties of the deterministic replicator dynamics \eqref{eq:RD}:
\begin{enumerate}
\item
Limits of interior orbits are Nash equilibria.
\item
Lyapunov stable states are Nash equilibria.
\item
Strict Nash equilibria are asymptotically stable under \eqref{eq:RD}.
\end{enumerate}

Of course, given the stochastic character of the dynamics \eqref{eq:SRD}, the notions of Lyapunov and asymptotic stability must be suitably modified.
In this \ac{SDE} context, we have:

\begin{definition}
\label{def:stability}
Let $\eq\in\strat$.
We will say that:
\begin{enumerate}
\item
$\eq$ is \emph{stochastically Lyapunov stable} under \eqref{eq:SRD} if, for every $\eps>0$ and for every neighborhood $U_{0}$ of $\eq$ in $\strat$, there exists a neighborhood $U\subseteq U_{0}$ of $\eq$ such that
\begin{equation}
\label{eq:stable-Lyap}
\prob(\text{$X(t)\in U_{0}$ for all $t\geq0$})
	\geq 1-\eps
	\quad
	\text{whenever $X(0) \in U$}.
\end{equation}

\item
$\eq$  is \emph{stochastically asymptotically stable} under \eqref{eq:SRD} if it is stochastically stable and attracting:
for every $\eps>0$ and for every neighborhood $U_{0}$ of $\eq$ in $\strat$, there exists a neighborhood $U\subseteq U_{0}$ of $\eq$ such that
\begin{equation}
\label{eq:stable-asym}
\txs
\prob\left(
	\text{$X(t)\in U_{0}$ for all $t\geq0$ and $\dis\lim_{t\to\infty} X(t) = \eq$}
	\right)
	\geq 1-\eps
	\quad
	\text{whenever $X(0) \in U$}.
\end{equation}
\end{enumerate}
\end{definition}



For \eqref{eq:SRD}, we have:

\begin{theorem}
\label{thm:folk}
Let $X(t)$ be an interior solution orbit of the stochastic replicator dynamics \eqref{eq:SRD} and let $\eq\in\strat$.
\begin{enumerate}
[\indent\upshape(1)]
\item
If $\prob\left(\lim_{t\to\infty} X(t) = \eq\right) > 0$, then $\eq$ is a Nash equilibrium of the noise-adjusted game $\modgame$.
\item
If $\eq$ is stochastically Lyapunov stable, then it is also a Nash equilibrium of the noise-adjusted game $\modgame$.
\item
If $\eq$ is a strict Nash equilibrium of the noise-adjusted game $\modgame$, then it is stochastically asymptotically stable under \eqref{eq:SRD}.
\end{enumerate}
\end{theorem}

\begin{remark}
By the nature of the modified payoff functions \eqref{eq:pay-mod}, strict equilibria of the original game $\game$ are also strict equilibria of $\modgame$, so Theorem \ref{thm:folk} implies that strict equilibria of $\game$ are also stochastically asymptotically stable under the stochastic dynamics \eqref{eq:SRD}. 
The converse does not hold: if the noise coefficients $\sigma_{\alpha}$ are sufficiently large, \eqref{eq:SRD} possesses stochastically asymptotically stable states that are not Nash equilibria of $\game$.
This is consistent with the behavior of \eqref{eq:SRD} in the pure noise case that we discussed in the previous section:
if $X(t)$ starts within $\eps$ of a vertex of $\strat$ and there are no payoff differences, then $X(t)$ converges to this vertex with probability at least $1-\eps$. 
\end{remark}


\begin{remark}
The condition for $\alpha$ to be a strict equilibrium of the modified game is that 
\begin{equation}
\label{eq:strict-mod}
\payv_{\beta} - \payv_{\alpha}
	< \frac{1}{2} \left(\sigma_{\alpha}^{2} + \sigma_{\beta}^{2} \right)
	\quad
	\text{for all $\beta\neq\alpha$,}
\end{equation}
where the payoffs and the noise coefficients are evaluated at the vertex $\bvec_{\alpha}$ of $\strat$ (note the similarity with \eqref{eq:dom-cond-const}).
To provide some intuition for this condition, consider the case of only two pure strategies, 
$\alpha$ and $\beta$, and assume constant noise coefficients. 
Letting $X(t)=X_{\beta}(t)$ and proceeding as in \eqref{eq:pn2strat}, we get $dX= X(1-X) \left[ (\payv_{\beta} - \payv_{\alpha}) \dd t - \sigma \dd W \right]$ where $\sigma^2= \sigma_{\alpha}^2 + \sigma_{\beta}^2$ and $W$ is a rescaled Wiener process.
Heuristically, a discrete-time counterpart of $X(t)$ is then provided by the random walk: 
\begin{equation}
X(n+1) - X(n)
	= X(n)(1-X(n)) \left[ (\payv_{\beta} - \payv_{\alpha}) \delta + \sigma \xi_{n} \sqrt{\delta} \right]
\end{equation}
where $\xi_{n} \in \{+1, -1\}$ is a zero-mean Bernoulli process, and the noise term is multiplied by $\sqrt{\delta}$ instead of $\delta$ because $dW\cdot \dd W = dt$.
For small $X$ and $\delta$, a simple computation then shows that, in the event $\xi_{n+1} = - \xi_{n}$, we have:
\begin{equation}
X(n+2) - X(n)
	= 2 \delta X(n) \left[ \payv_{\beta} - \payv_{\alpha} - \tfrac{1}{2}\sigma^{2} \right]
	+ o(\delta) + o(X(n)).
\end{equation} 
Since $\sigma^2= \sigma_{\alpha}^2 + \sigma_{\beta}^2$, the bracket is negative (so $X_{\alpha}= 1- X$ increases) if and only if condition \eqref{eq:strict-mod} is satisfied.
Thus, \eqref{eq:strict-mod} may be interpreted as saying that when the discrete-time process $X(n)$ is close to $\bvec_{\alpha}$ and the random noise term $\xi_{n}$ takes two successive steps in opposite direction, then the process ends up even closer to $\bvec_{\alpha}$.%
\footnote{Put differently, it's more probable for $X(n)$ to decrease rather than increase:
$X(n+2) > X(n)$ with probability $1/4$ (i.e. if and only if $\xi_{n}$ takes two positive steps), while $X(n+2) < X(n)$ with probability $3/4$.}
On the other hand, if the opposite strict inequality holds, then this interpretation suggests that $\beta$ should successfully invade a population where most individuals play $\alpha$ \textendash\ which, in turn, explains \eqref{eq:dom-cond-const}.
\end{remark}


\begin{proof}[Proof of Theorem \ref{thm:folk}]
Contrary to the approach of \cite{HI09}, we will not employ the stochastic Lyapunov method \citep[see e.g.][]{Kha12} which requires calculating the infinitesimal generator of \eqref{eq:SRD}.
Instead, motivated by the recent analysis of \cite{BM14}, our proof will rely on the ``dual'' variables $Y_{\alpha} = \log X_{\alpha}$ that were already used in the proof of Theorem \ref{thm:dominated}.

\paragraph{Part 1.}
We argue by contradiction.
Indeed, assume that $\prob\left(\lim_{t\to\infty} X(t) = \eq\right) > 0$ but that $\eq$ is not Nash for the noise-adjusted game $\modgame$, so $\modpayv_{\alpha}(\eq) < \modpayv_{\beta}(\eq)$ for some $\alpha\in\supp(\eq)$, $\beta\in\act$. 
On that account, let $U$ be a sufficiently small neighborhood of $\eq$ in $\strat$ such that $\modpayv_{\beta}(x) - \modpayv_{\alpha}(x) \geq m$ for some $m > 0$ and for all $x\in U$.
Then, by \eqref{eq:dY2}, we get:
\begin{equation}
\label{eq:Ydiff1}
\begin{aligned}
dY_{\alpha} - dY_{\beta}
	&= \left[\payv_{\alpha} - \payv_{\beta}\right] dt
	- \frac{1}{2} \left[ (1-2X_{\alpha}) \sigma_{\alpha}^{2} - (1-2X_{\beta}) \sigma_{\beta}^{2} \right] dt
	\\
	&+ \sigma_{\alpha} \dd W_{\alpha} - \sigma_{\beta} \dd W_{\beta},
\end{aligned}
\end{equation}
so, if $X(t)$ is an interior orbit of \eqref{eq:SRD} that converges to $\eq$, we will have:
\begin{equation}
\label{eq:Ydiff2}
dY_{\alpha} - dY_{\beta}
	\leq - m \dd t - d\xi
	\quad
	\text{for all large enough $t>0$},
\end{equation}
where $\xi$ denoting the martingale part of \eqref{eq:Ydiff1}.
Since the diffusion coefficients of \eqref{eq:Ydiff1} are bounded, Lemma \ref{lem:Wbound} shows that $mt + \xi(t) \sim mt$ for large $t$ (a.s.), so
\begin{equation}
\log\frac{X_{\alpha}(t)}{X_{\beta}(t)}
	\leq \log\frac{X_{\alpha}(0)}{X_{\beta}(0)} - mt - \xi(t)
	\sim - mt
	\to -\infty
	\quad
	\text{(a.s.)}
\end{equation}
as $t\to\infty$. 
This implies that $\lim_{t\to\infty} X_{\alpha}(t) = 0$, contradicting our original assumption that $X(t)$ stays in a small enough neighborhood of $\eq$ with positive probability (recall that $\eq_{\alpha}>0$);
we thus conclude that $\eq$ is a Nash equilibrium of the noise-adjusted game $\modgame$, as claimed.

\paragraph{Part 2.}
Assume that $\eq$ is stochastically Lyapunov stable.
Then, every neighborhood $U$ of $\eq$ admits an interior trajectory $X(t)$ that stays in $U$ for all time with positive probability.
The proof of Part 1 shows that this only possible if $\eq$ is a Nash equilibrium of the modified $\modgame$, so our claim follows.

\paragraph{Part 3.}
To show that strict Nash equilibria of $\modgame$ are stochastically asymptotically stable, let $\eq = (\alpha_{1}^{\ast},\dotsc, \alpha_{N}^{\ast}) \in\strat$ be a strict equilibrium of $\modgame$.
Then, suppressing the population index $k$ as before, let
\begin{equation}
\label{eq:Zdef}
Z_{\alpha}
	= Y_{\alpha} - Y_{\alpha^{\ast}},
\end{equation}
so that $X(t)\to\eq$ if and only if $Z_{\alpha}(t) \to -\infty$ for all $\alpha\in\act^{\ast} \equiv \act\exclude{\alpha^{\ast}}$.%
\footnote{Simply note that $X_{\alpha^{\ast}} = \big(1 + \sum_{\beta\in\act^{\ast}} \exp(Z_{\beta})\big)^{-1}$.}

To proceed, fix some probability threshold $\eps>0$ and a neighborhood $U_{0}$ of $\eq$ in $\strat$. 
Since $\eq$ is a strict equilibrium of $\modgame$, there exists a neighborhood $U \subseteq U_{0}$ of $\eq$ and some $m>0$ such that
\begin{equation}
\modpayv_{\alpha^{\ast}}(x) - \modpayv_{\alpha}(x)
	\geq m
	\quad
	\text{for all $x\in U$ and for all $\alpha\in\act^{\ast}$.}
\end{equation}
Let $M>0$ be sufficiently large so that $X(t) \in U$ if $Z_{\alpha}(t) \leq -M$ for all $\alpha\in\act^{\ast}$;
we will show that if $M$ is chosen suitably (in terms of $\eps$) and $Z_{\alpha}(0) < -2M$, then $X(t) \in U$ for all $t\geq0$ and $Z_{\alpha}(t) \to -\infty$ with probability at least $1-\eps$, i.e. $\eq$ is stochastically asymptotically stable.

To that end, take $Z_{\alpha}(0) \leq -2M$ in \eqref{eq:Zdef} and define the first exit time:
\begin{equation}
\label{eq:hitU}
\tau_{U}
	= \inf\{t>0: X(t) \notin U\}.
\end{equation}
By applying \eqref{eq:Ydiff1}, we then get:
\begin{equation}
dZ_{\alpha}
	= dY_{\alpha} - dY_{\alpha^{\ast}}
	= \big[ \modpayv_{\alpha} - \modpayv_{\alpha^{\ast}} \big] \dd t
	- d\xi,
\end{equation}
where the martingale term $d\xi$ is defined as in \eqref{eq:Ydiff1}, taking $\beta = \alpha^{\ast}$.
Hence, for all $t\leq \tau_{U}$, we will have:
\begin{equation}
\label{eq:Z1}
Z_{\alpha}(t)
	= Z_{\alpha}(0)
	+ \int_{0}^{t} \left[ \modpayv_{\alpha}(X(s)) - \modpayv_{\alpha^{\ast}}(X(s)) \right] \dd s
	- \xi(t)
	\leq -2M - mt - \xi(t).
\end{equation}

By the time-change theorem for martingales \cite[Cor.~8.5.4]{Oks07}, there exists a standard Wiener process $\wilde W(t)$ such that $\xi(t) = \wilde W(\rho(t))$ where $\rho = [\xi,\xi]$ denotes the quadratic variation of $\xi$;
as such, we will have $Z_{\alpha}(t) \leq - M$ whenever $\wilde W(\rho(t)) \geq -M - mt$.
However, with $\sigma$ Lipschitz over $\strat$, we readily get $\rho(t) \leq Kt$ for some positive constant $K>0$, so it suffices to show that the hitting time
\begin{equation}
\label{eq:hitline}
\tau_{0}
	= \inf\big\{t>0: \wilde W(t) = -M - mt/K \big\}
\end{equation}
is finite with probability not exceeding $\eps$.
Indeed, if a trajectory of $\wilde W(t)$ has $\wilde W(t) \geq -M - mt/K$ for all $t\geq0$, we will also have
\begin{equation}
\wilde W(\rho(t))
	\geq -M - m \rho(t)/K
	\geq - M - mt,
\end{equation}
so $\tau_{U}$ is infinite for every trajectory of $\wilde W$ with infinite $\tau_{0}$, hence $\prob(\tau_{U}<+\infty) \leq \prob(\tau_{0}<+\infty)$.
Lemma \ref{lem:hitprob} then shows that $\prob(\tau_{0} < +\infty) = e^{-2Mm/K}$, so, if we take $M> - (2m)^{-1} K \log \eps$, we get $\prob(\tau_{U} = \infty) \geq 1 - \eps$.
Conditioning on the event $\tau_{U} = +\infty$, Lemma \ref{lem:Wbound} applied to \eqref{eq:Z1} yields
\begin{equation}
\label{eq:Z2}
Z_{\alpha}(t)
	\leq -2M - mt - \xi(t)
	\sim - mt
	\to-\infty
	\quad
	\text{(a.s.)}
\end{equation}
so $X(t) \to \eq$ with probability at least $1-\eps$, as was to be shown.
\end{proof}

\begin{remark}
\label{rem:modified}
As mentioned before, \cite{HI09} state a similar ``evolutionary folk theorem'' in the context of single-population random matching games under the stochastic replicator dynamics with aggregate shocks \eqref{eq:ASRD}.
In particular, \cite{HI09} consider the modified game:
\begin{equation}
\label{eq:pay-mod-AS}
\modpayv_{\alpha}(x)
	= \payv_{\alpha}(x) - \frac{1}{2} \sigma_{\alpha}^{2},
\end{equation}
where $\sigma_{\alpha}$ denotes the intensity of the aggregate shocks in \eqref{eq:ASRD}, and they show that strict Nash equilibria of this noise-adjusted game are stochastically asymptotically stable under \eqref{eq:ASRD}.
It is  interesting to note that the adjustments \eqref{eq:pay-mod} and \eqref{eq:pay-mod-AS} do not coincide: the payoff shocks affect the deterministic replicator equation \eqref{eq:RD} in a different way than the aggregate shocks of \eqref{eq:ASRD}.
Heuristically, in the model of \cite{FH92}, noise is detrimental because for a given expected growth rate, noise almost surely lowers the long-term average geometric growth rate of the total number of individuals playing $\alpha$ by the quantity $\frac{1}{2} \sigma^{2}_{\alpha}$.
In a geometric growth process, the quantities that matter (the proper fitness measures) are these long-term geometric growth rates, so the relevant payoffs are those of this modified game.%
\footnote{In a discrete time setting, if $Z(n+1)= g(n) Z_n$ and $g(n)=k_i$ with probability $p_i$, what we mean is that the quantity that a.s. governs the long-term growth of $Z$ is not $E(g)=\sum_{i} p_i k_i$, but $\exp (E (\ln g))= \prod_i k_i^{p_i}$.}
In our model, noise is not detrimental, but if it is strong enough compared to the deterministic drift, then, with positive probability, it may lead to other outcomes than the deterministic model.
Instead, the assumptions of Theorems \ref{thm:dominated} and \ref{thm:folk} should be interpreted as guaranteeing that the deterministic drift prevails.
One way to see this is to note that if $\beta$ strictly dominates $\alpha$ in the original game and both strategies are affected by the same noise intensity ($\sigma_{\alpha}^{2} = \sigma_{\beta}^{2} = \sigma^{2}$), then $\beta$ need not dominate $\alpha$ in the modified game defined by \eqref{eq:pay-mod}, unless the payoff margin in the original game is always greater than $\sigma^2$. 
\end{remark}

\begin{remark}
It is also worth contrasting Theorem \ref{thm:folk} to the unconditional convergence and stability results of \cite{MM10} for the stochastic replicator dynamics of exponential learning \eqref{eq:SXRD}.
As in the case of dominated strategies, the reason for this qualitative difference is the distinct origins of the perturbation process:
the Itô correction in \eqref{eq:SXRD} is ``just right'' with respect to the dual variables $Y_{\alpha} = \log X_{\alpha}$, so a state $\eq\in\strat$ is stochastically asymptotically stable under the \eqref{eq:SXRD} if and only if it is a strict equilibrium of the original game $\game$.
 \end{remark}

\section{The effect of aggregating payoffs}
\label{sec:cumulative}

In this section, we examine the case where players are less ``myopic'' and, instead of using revision protocols driven by their instantaneous payoffs, they base their decisions on the cumulative payoffs of their strategies over time.
Formally, focusing for concreteness on the ``imitation of success'' revision protocol \eqref{eq:success}, this amounts to considering conditional switch rates of the form:
\begin{equation}
\label{eq:success-cum}
\tilde\rho_{\alpha\beta}
	= x_{\beta} U_{\beta},
\end{equation}
where
\begin{equation}
\label{eq:pay-cum}
U_{\beta}(t)
	= \int_{0}^{t} \payv_{\beta}(x(s)) \dd s
\end{equation}
denotes the cumulative payoff of strategy $\beta$ up to time $t$. 
In this case, \eqref{eq:RD} becomes:
\begin{equation}
\label{eq:RD-cum}
\dot x_{\alpha}
	= x_{\alpha} \left[ U_{\alpha} - \insum_{\beta} x_{\beta}U_{\beta} \right],
\end{equation}
and, as was shown by \cite{LM13}, the evolution of mixed strategy shares is governed by the (deterministic) \emph{second order replicator dynamics}:
\begin{equation}
\label{eq:RD-2}
\tag{RD$_{2}$}
\ddot x_{\alpha}
	= x_{\alpha} \left[ \payv_{\alpha}(x) - \insum_{\beta} x_{\beta}\pay_{\beta}(x) \right]
	+ x_{\alpha} \left[ \dot x_{\alpha}^{2}/x_{\alpha}^{2} - \insum_{\beta} \dot x_{\beta}^{2} / x_{\beta} \right].
\end{equation}

As in the previous section, we are interested in the effects of random payoff shocks on the dynamics \eqref{eq:RD-2}. 
If the game's payoff functions are subject to random shocks at each instant in time, then these shocks will also be aggregated over time, leading to the perturbed cumulative payoff process:
\begin{equation}
\label{eq:pay-cum-noise}
\hat U_{\alpha}(t)
	= \int_{0}^{t} \payv_{\alpha}(X(s)) \dd s
	+ \int_{0}^{t} \sigma_{\alpha}(X(s)) \dd W_{\alpha}(s).
\end{equation}
Since $\hat U_{\alpha}$ is continuous (a.s.), we obtain the stochastic integro-differential dynamics:
\begin{equation}
\label{eq:SRD-cum}
\begin{aligned}
\dot X_{\alpha}
	&= X_{\alpha} \left[
	U_{\alpha}(t) - \insum_{\beta} X_{\beta}(t) U_{\beta}(t)
	\right]
	\\
	&+ X_{\alpha} \left[
	\int_{0}^{t} \sigma_{\alpha}(X(s)) \dd W_{\alpha}(s)
	- \insum_{\beta} \int_{0}^{t} X_{\beta}(s) \sigma_{\beta}(X(s)) \dd W_{\beta}(s)
	\right],
\end{aligned}
\end{equation}
where, as in \eqref{eq:SRD}, we assume that the Brownian disturbances $W_{\alpha}(t)$ are independent.

To obtain an autonomous \ac{SDE} from \eqref{eq:SRD-cum}, let $V_{\alpha} = \dot X_{\alpha}$ denote the growth rate of strategy $\alpha$.
Then, differentiating \eqref{eq:SRD-cum} yields:
\begin{subequations}
\label{eq:dV1}
\begin{flalign}
dV_{\alpha}
	&\label{eq:dV-det}
	= X_{\alpha} \left[ \dot U_{\alpha} - \insum_{\beta} X_{\beta} \dot U_{\beta} \right] \dd t
	\\
	&\label{eq:dV-speed}
	+ V_{\alpha} \left[ U_{\alpha} - \insum_{\beta} X_{\beta} U_{\beta} \right] dt
	- X_{\alpha} \insum_{\beta} U_{\beta} V_{\beta} \dd t
	\\
	&\label{eq:dV-intnoise}
	+ V_{\alpha} \left[
	\int_{0}^{t} \sigma_{\alpha}(X(s)) \dd W_{\alpha}(s)
	- \insum_{\beta} \int_{0}^{t} X_{\beta}(s) \sigma_{\beta}(X(s)) \dd W_{\beta}(s)
	\right] dt
	\\
	&\label{eq:dV-noise}
	+ X_{\alpha} \left[ \sigma_{\alpha}(X) \dd W_{\alpha} - \insum_{\beta} \sigma_{\beta}(X) X_{\beta} \dd W_{\beta} \right].
\end{flalign}
\end{subequations}
By \eqref{eq:SRD-cum}, the sum of the first term of \eqref{eq:dV-speed} and \eqref{eq:dV-intnoise} is equal to $V_{\alpha}^{2}/X_{\alpha}$.
Thus, using \eqref{eq:pay-cum} we obtain:
\begin{flalign}
dV_{\alpha}
	&= X_{\alpha} \left[ \payv_{\alpha}(X) - \insum_{\beta} X_{\beta} \payv_{\beta}(X) \right] dt
	+ \frac{V_{\alpha}^{2}}{X_{\alpha}} dt - X_{\alpha} \insum_{\beta} U_{\beta} V_{\beta} \dd t
	\notag\\
	&+ X_{\alpha} \left[ \sigma_{\alpha}(X) \dd W_{\alpha} - \insum_{\beta} \sigma_{\beta}(X) X_{\beta} \dd W_{\beta} \right],
\end{flalign}
and, after summing over all $\alpha$ and solving for $X_{\alpha} \insum_{\beta} U_{\beta} V_{\beta} \dd t$, we get the second order \ac{SDE} system:%
\footnote{Recall that $\insum_{\alpha} dV_{\alpha} = 0$ since $\insum_{\alpha} X_{\alpha} = 1$.}
\begin{flalign}
\label{eq:SRD-2}
dX_{\alpha}
	&= V_{\alpha} \dd t
	\notag\\
dV_{\alpha}
	&= X_{\alpha} \left[ \payv_{\alpha}(X) - \insum_{\beta} x_{\beta} \payv_{\beta}(X) \right] dt
	+ X_{\alpha} \left[ V_{\alpha}^{2}/X_{\alpha}^{2} - \insum_{\beta} V_{\beta}^{2}/X_{\beta} \right] dt
	\\
	&+ X_{\alpha} \left[ \sigma_{\alpha}(X) \dd W_{\alpha} - \insum_{\beta} \sigma_{\beta}(X) X_{\beta} \dd W_{\beta} \right].
	\notag
\end{flalign}

By comparing the second order system \eqref{eq:SRD-2} to \eqref{eq:RD-2}, we see that there is no Itô correction, just as in the first order case.%
\footnote{The reason however is different:
in \eqref{eq:SRD}, there is no Itô correction because the noise is added directly to the dynamical system under study;
in \eqref{eq:SRD-2}, there is no Itô correction because the noise is integrated over, so $X_{\alpha}$ is smooth (and, hence, obeys the rules of ordinary calculus).}
By using similar arguments as in \cite{LM13}, it is then possible to show that the system \eqref{eq:SRD-2} is well-posed, i.e. it admits a unique (strong) solution $X(t)$ for every interior initial condition $X(0)\in\intstrat$, $V(0)\in\R^{\act}$ and this solution remains in $\intstrat$ for all time.


With this well-posedness result at hand, we begin by showing that \eqref{eq:SRD-2} eliminates strategies that are dominated in the original game $\game$ (instead of the modified game $\modgame$):

\begin{theorem}
\label{thm:dominated-cum}
Let $X(t)$ be a solution orbit of the dynamics \eqref{eq:SRD-cum} and assume that $\alpha\in\act$ is dominated by $\beta\in\act$.
Then, $\alpha$ becomes extinct \textup(a.s.\textup).
\end{theorem}

\begin{proof}
As in the proof of Theorem \ref{thm:dominated},
let $Y_{\alpha} = \log X_{\alpha}$.
Then, following the same string of calculations leading to \eqref{eq:dY2}, we get:%
\footnote{Recall that $\int_{0}^{t} \sigma_{\alpha}(X(s)) \dd W_{\alpha}(s)$ is continuous, so the only Itô correction stems from random mutations.}
\begin{subequations}
\label{eq:dYdiff-cum}
\begin{flalign}
dY_{\alpha} - dY_{\beta}
	&\label{eq:dY-drift-pay}
	= \left[ U_{\alpha} - U_{\beta} \right] dt
	\\
	&\label{eq:dY-drift-noise}
	+ \left[
	\int_{0}^{t} \sigma_{\alpha}(X) \dd W_{\alpha} - \int_{0}^{t} \sigma_{\beta}(X) \dd W_{\beta}
	\right] dt
\end{flalign}
\end{subequations}
Since $\alpha$ is dominated by $\beta$, there exists some positive $m>0$ such that $\payv_{\alpha} - \payv_{\beta} \leq - m$, and hence $U_{\alpha}(t) - U_{\beta}(t) \leq - mt$.
Furthermore, with $\sigma$ bounded and continuous on $\strat$, Lemma \ref{lem:Wbound} readily yields:
\begin{equation}
-mt
	+ \left[
	\int_{0}^{t} \sigma_{\alpha}(X) \dd W_{\alpha} - \int_{0}^{t} \sigma_{\beta}(X) \dd W_{\beta}
	\right]
	\sim - mt
\end{equation}
as $t\to\infty$.
Accordingly, \eqref{eq:dYdiff-cum} becomes:
\begin{equation}
\label{eq:dYdiff-cum2}
dY_{\alpha} - dY_{\beta}
	\leq - mt dt
	+ \theta(t) dt
\end{equation}
where the remainder function $\theta(t)$ corresponding to the drift term \eqref{eq:dY-drift-noise} is sublinear in $t$.
By integrating and applying Lemma \ref{lem:Wbound} a second time, we then obtain:
\begin{flalign}
\label{eq:dYdiff-cum3}
Y_{\alpha}(t) - Y_{\beta}(t)
	&\leq Y_{\alpha}(0) - Y_{\beta}(0)
	 - \frac{1}{2} mt^{2}
	 + \int_{0}^{t} \theta(s) \dd s
	 \sim - \frac{1}{2} mt^{2}
	 \quad
	 \text{(a.s.)}.
\end{flalign}
We infer that $\lim_{t\to\infty} Y_{\alpha}(t) = 0$ (a.s.), i.e. $\alpha$ becomes extinct along $X(t)$.
\end{proof}

\begin{remark}
In view of Theorem \ref{thm:dominated-cum}, we see that the ``imitation of long-term success'' protocol \eqref{eq:success-cum} provides more robust elimination results than \eqref{eq:success} in the presence of payoff shocks:
contrary to Theorem \ref{thm:dominated}, there are no ``small noise'' requirements in Theorem \ref{thm:dominated-cum}.%
\footnote{Theorem \ref{thm:dominated-cum} actually applies to mixed dominated strategies as well (even iteratively dominated ones).
The proof is a simple adaptation of the pure strategies case, so we omit it.}
\end{remark}

Our next result provides the analogue of Theorem \ref{thm:folk} regarding the stability of equilibrium play:

\begin{theorem}
\label{thm:folk-cum}
Let $X(t)$ be an interior solution orbit of the stochastic dynamics \eqref{eq:SRD-cum} and let $\eq \in \strat$.
Then:
\begin{enumerate}
\item
If $\prob\left(\lim_{t\to\infty} X(t) = \eq\right) > 0$, $\eq$ is a Nash equilibrium of $\game$.
\end{enumerate}
Moreover, for every neighborhood $U_{0}$ of $\eq$ and for all $\eps>0$, we have:
\begin{enumerate}
\setcounter{enumi}{1}
\item
If
\(
\prob(\text{$X(t)\in U_{0}$ for all $t\geq0$})
	\geq 1-\eps
	\text{ whenever $X(0) \in U$}
\)
for some neighborhood $U\subseteq U_{0}$ of $\eq$, then $\eq$ is a Nash equilibrium of $\game$.
\item
If $\eq$ is a strict Nash equilibrium of $\game$, there exists a neighborhood $U$ of $\eq$ such that:
\begin{equation}
\prob\left(\text{$X(t) \in U_{0}\,$ for all $t\geq0$ and $\lim_{t\to\infty} X(t) = \eq$}\right)
	\geq 1- \eps,
\end{equation}
whenever $X(0)\in U$.
\end{enumerate}
\end{theorem}

\begin{remark}
Part 1 of Theorem \ref{thm:folk-cum} is in direct analogy with Part 1 of Theorem \ref{thm:folk}:
the difference is that Theorem \ref{thm:folk-cum} shows that only Nash equilibria of the original game $\game$ can be $\omega$-limits of interior orbits with positive probability.
Put differently, if $\eq$ is a strict equilibrium of $\modgame$ but not of $\game$,%
\footnote{Recall here that strict equilibria of $\game$ are also strict equilibria of $\modgame$, but the converse need not hold.}
there is zero probability that \eqref{eq:SRD-cum} converges to $\eq$.

On the other hand, Parts 2 and 3 are not tantamount to stochastic stability (Lyapunov or asymptotic) under the autonomous \ac{SDE} system \eqref{eq:SRD-2}.
The difference here is that \eqref{eq:SRD-2} is only well-defined in the interior of $\intstrat$, so it is not straightforward how to define the notion of (stochastic) stability for boundary points $\eq\in\bd(\strat)$;
moreover, given that \eqref{eq:SRD-2} is a second order system, stability should be stated in terms of the problem's entire phase space, including initial velocities (for a relevant discussion, see \citealp{LM13}).
Instead, the stated stability conditions simply reflect the fact that the integro-differential dynamics \eqref{eq:SRD-cum} always start with initial velocity $V(0) = 0$, so this added complexity is not relevant.
\end{remark}

\begin{proof}[Proof of Theorem \ref{thm:folk-cum}]
We shadow the proof of Theorem \ref{thm:folk}.

\paragraph{Part 1.}
Assume that $\prob\left(\lim_{t\to\infty} X(t) = \eq\right) > 0$ for some $\eq\in\strat$.
If $\eq$ is not a Nash equilibrium of $\game$, we will have $\payv_{\alpha}(\eq) < \payv_{\beta}(\eq)$ for some $\alpha\in\supp(\eq)$, $\beta\in\act$. 
Accordingly, let $U$ be a sufficiently small neighborhood of $\eq$ in $\strat$ such that $\payv_{\beta}(x) - \payv_{\alpha}(x) \geq m$ for some $m > 0$ and for all $x\in U$.
Since $X(t)\to\eq$ with positive probability, it also follows that $\prob(X(t)\in U \text{ for all } t\geq0) >0$;
hence, arguing as in \eqref{eq:dYdiff-cum3} and conditioning on the positive probability event ``$X(t)\in U$ for all $t\geq0$'', we get:
\begin{equation}
\label{eq:Ydiff-cum4}
Y_{\alpha}(t) - Y_{\beta}(t)
	 \sim - \frac{1}{2} mt^{2}
	 \quad
	 \text{(conditionally a.s.).}
\end{equation}
This implies $X_{\alpha}(t) \to 0$, contradicting our original assumption that $X(t)$ stays in a small neighborhood of $\eq$.
We infer that $\eq$ is a Nash equilibrium of $\game$, as claimed.

\paragraph{Part 2.}
Simply note that the stability assumption of Part 2 implies that there exists a positive measure of interior trajectories $X(t)$ that remain in an arbitrarily small neighborhood of $\eq$ with positive probability.
The proof then follows as in Part 1.

\paragraph{Part 3.}
Let $Z_{\alpha} = Y_{\alpha} - Y_{\alpha^{\ast}}$ be defined as in \eqref{eq:Zdef} and let $m>0$ be such that $\payv_{\alpha^{\ast}}(x) - \payv_{\alpha}(x) \geq m$ for all $x$ in some sufficiently small neighborhood of $\eq$.
Also, let $M>0$ be sufficiently large so that $X(t) \in U$ if $Z_{\alpha}(t) \leq -M$ for all $\alpha\in\act^{\ast}$;
as in the proof of Theorem \ref{thm:folk}, we will show that there is a suitable choice of $M$ such that $Z_{\alpha}(0) < -2M$ for all $\alpha\neq\alpha^{\ast}$ implies that $X(t) \in U$ for all $t\geq0$ and $Z_{\alpha}(t) \to -\infty$ for all $\alpha\neq\alpha^{\ast}$ with probability at least $1-\eps$.

Indeed, by setting $\beta=\alpha^{\ast}$ in \eqref{eq:dYdiff-cum}, we obtain:
\begin{flalign}
\label{eq:dZ-cum}
dZ_{\alpha}
	&= \left[ U_{\alpha} - U_{\alpha^{\ast}} \right] dt
	+ \left[
	\int_{0}^{t} \sigma_{\alpha}(X) \dd W_{\alpha} - \int_{0}^{t} \sigma_{\alpha^{\ast}}(X) \dd W_{\alpha^{\ast}}
	\right] dt
\end{flalign}
so, recalling ({eq:pay-cum-noise}), for all $t\leq \tau_{U} = \inf\{t>0: X(t) \notin U\}$, we will have:
\begin{equation}
\label{eq:dZ-cum2}
Z_{\alpha}(t)
	\leq -2M - \frac{1}{2} mt^{2} + \int_{0}^{t} \theta(s) \dd s - \xi(t),
\end{equation}
where $\xi$ denotes the martingale part of \eqref{eq:dZ-cum} and $\theta(t)$ is defined as in \eqref{eq:dYdiff-cum2}.

Now, let $W(t)$ be a Wiener process starting at the origin.
We will show that if $M$ is chosen sufficiently large, then
\begin{equation}
\label{eq:hitprob2}
\prob\left(
	\txs
	M + \frac{1}{2} mt^{2}
	\geq \int_{0}^{t} W(s) \dd s
	\;
	\text{for all $t\geq0$}
	\right)
	\geq 1-\eps.
\end{equation}
With a fair degree of hindsight, we note first that $\frac{1}{4}mt^{2} + \frac{1}{2}M \geq a t^{2} + bt + c$ where $a=\frac{1}{4}m$, $b = \frac{1}{2}\sqrt{Mm}$ and $c = \frac{1}{2}M$, so it suffices to show that the hitting time $\tau = \inf\{t: \int_{0}^{t} W(s) \dd s = at^{2} + bt + c\}$ is infinite with probability at least $1-\eps$.
However, by the mean value theorem, there exists some (random) time $\tau_{1}$ such that:
\begin{equation}
2a\tau_{1} + b - W(\tau_{1})
	= \frac{0 - c}{\tau}
	\leq 0.
\end{equation}
Since $c/\tau > 0$, the hitting time $\tau' = \inf\{t>0: W(t) = 2at + b\}$ will satisfy $\tau'(\omega) < \tau(\omega)$ for every trajectory $\omega$ of $W$ with $\tau(\omega) < \infty$.
However, Lemma \ref{lem:hitprob} gives $\prob[\tau'<\infty] = \exp(-2ab)$, hence:
\begin{equation}
\prob(\tau<\infty)
	\leq \prob(\tau'<\infty)
	= \exp(-2ab)
	= \exp(-Mm/4),
\end{equation}
i.e. $\prob(\tau<\infty)$ can be made arbitrarily small by choosing $M$ large enough.
We thus deduce that
\begin{equation}
\int_{0}^{t} W(s) \dd s
	\leq \frac{1}{2} mt^{2} + M
	\quad
	\text{for all $t\geq0$}
\end{equation}
with probability no less than $1-\eps$.

Going back to \eqref{eq:dZ-cum2}, we see that $\int_{0}^{t}\theta(s) \dd s - \xi(t) - \frac{1}{2}mt^{2}$ remains below $M$ for all time with probability at least $1-\eps$ (simply use the probability estimate \eqref{eq:hitprob2} and argue as in the proof of Theorem \ref{thm:folk} recalling that the processes $W_{\alpha}$ are assumed independent).
In turn, this shows that $\prob(X(t) \in U \text{ for all } t\geq0) \geq 1-\eps$, so, conditioning on this last event and letting $t\to\infty$ in \eqref{eq:dZ-cum2}, we obtain:
\begin{equation}
\prob\left(\lim_{t\to\infty} Z_{\alpha}(t) = -\infty \,\Big\vert\, X(t) \in U \text{ for all } t\geq0 \right)
	= 1
	\quad
	\text{for all $\alpha\neq\alpha^{\ast}$.}
\end{equation}
We conclude that $X(t)$ remains in $U_{0}$ and $X(t) \to \eq$ with probability at least $1-\eps$, as was to be shown.
\end{proof}


\section{Discussion}
\label{sec:discussion}

In this section, we discuss some points that would have otherwise disrupted the flow of the main text:

%
%
%
%

\subsection{Payoff shocks in bimatrix games}
\label{sec:matching}

Throughout our paper, we have worked with generic population games,
so we have not made any specific assumptions on the payoff shocks either.
On the other hand, if the game's payoff functions are obtained from some common underlying structure, then the resulting payoff shocks may also end up having a likewise specific form.

For instance, consider a basic (symmetric) random matching model where pairs of players are drawn randomly from a nonatomic population to play a symmetric two-player game with payoff matrix $V_{\alpha\beta}$, $\alpha,\beta=1,\dotsc,n$.
In this case, the payoff to an $\alpha$-strategist in the population state $x\in\strat$ will be of the form:
\begin{equation}
\label{eq:pay-matching}
\payv_{\alpha}(x)
	= \insum_{\beta} V_{\alpha\beta} x_{\beta}.
\end{equation}
Thus, if the entries of $V$ are disturbed at each $t\geq0$ by some (otherwise independent) white noise process $\xi_{\alpha\beta}$, the perturbed payoff matrix $\hat V_{\alpha\beta} = V_{\alpha\beta} + \xi_{\alpha\beta}$ will result in the total payoff shock:
\begin{equation}
\label{eq:noise-matching}
\xi_{\alpha}
	= \insum_{\beta} \xi_{\alpha\beta} x_{\beta}.
\end{equation}
The stochastic dynamics \eqref{eq:SRD-Langevin} thus become:
\begin{equation}
\label{eq:SRD-corr}
\begin{aligned}
dX_{\alpha}
	&= X_{\alpha} \left[ \insum_{\beta} V_{\alpha\beta} X_{\beta} - \insum_{\beta,\gamma}  V_{\beta\gamma} X_{\beta} X_{\gamma} \right] dt
	\\
	&+ X_{\alpha} \left[ \insum_{\beta} \sigma_{\alpha\beta} X_{\beta} \dd W_{\alpha\beta} - \insum_{\beta,\gamma} \sigma_{\beta\gamma} X_{\beta} X_{\gamma} \dd W_{\beta\gamma} \right],
\end{aligned}
\end{equation}
where the Wiener processes $W_{\alpha\beta}$ are assumed independent.

To compare \eqref{eq:SRD-corr} with the core model \eqref{eq:SRD}, the same string of calculations as in the proof of Theorems \ref{thm:dominated} and \ref{thm:folk} leads to the modified payoff functions:
\begin{equation}
\label{eq:pay-mod-corr}
\modpayv_{\alpha}(x)
	= \payv_{\alpha}(x) - \frac{1}{2} (1- 2x_{\alpha}) \insum_{\beta} \sigma_{\alpha\beta}^{2} x_{\beta}^{2}.
\end{equation}
It is then trivial to see that Theorems \ref{thm:dominated} and \ref{thm:folk} still apply with respect to the modified game $\modgame$ with payoff functions defined as above; 
however, seeing as \eqref{eq:pay-mod-corr} is cubic in $x_{\alpha}$ and considering the case of constant noise, these modified payoff functions no longer correspond to random matching in a modified bimatrix game.

\subsection{Stratonovich-type perturbations}
\label{sec:Stratonovich}

Depending on the origins of the payoff shock process (for instance, if there are nontrivial autocorrelation effects that do not vanish in the continuous-time regime), the perturbed dynamics \eqref{eq:SRD} could instead be written as a Stratonovich-type \ac{SDE} \citep{Kuo06}:
\begin{equation}
\label{eq:SRD-Strat}
\pd X_{\alpha}
	= X_{\alpha} \left[ \payv_{\alpha}(X) - \insum_{\beta} X_{\beta} \payv_{\beta}(X) \right] dt
	+ X_{\alpha} \left[ \sigma_{\alpha} \,\pd W_{\alpha} - \insum_{\beta} X_{\beta} \sigma_{\beta} \,\pd W_{\beta} \right],
\end{equation}
where $\pd(\argdot)$ denotes Stratonovich integration.%
\footnote{For a general overview of the differences between Itô and Stratonovich integration, see \cite{vK81};
for a more specific account in the context of stochastic population growth models, the reader is instead referred to \cite{KP06} and \cite{HI09}.}
In this case, if $M_{\alpha\beta} = X_{\alpha}(\delta_{\alpha\beta} - X_{\beta}) \sigma_{\beta}$ denotes the diffusion matrix of \eqref{eq:SRD-Strat}, the Itô equivalent \ac{SDE} corresponding to \eqref{eq:SRD-Strat} will be:
\begin{equation}
\label{eq:SRD-Ito}
\begin{aligned}
dX_{\alpha}
	&= X_{\alpha} \left[ \payv_{\alpha}(X) - \insum_{\beta} X_{\beta} \payv_{\beta}(X) \right] dt
	+ \frac{1}{2} \insum_{\beta,\gamma} \frac{\pd M_{\alpha\beta}}{\pd X_{\gamma}} M_{\gamma\beta} \dd t
	\\
	&+ X_{\alpha} \left[ \sigma_{\alpha} \dd W_{\alpha} - \insum_{\beta} X_{\beta} \sigma_{\beta} \dd W_{\beta} \right].
\end{aligned}
\end{equation}
Then, assuming that the shock coefficients $\sigma_{\beta}$ are constant, some algebra yields the following explicit expression for the Itô correction of \eqref{eq:SRD-Ito}:
\begin{flalign}
\label{eq:correction}
\frac{1}{2}	 \insum_{\beta,\gamma}
	&\frac{\pd M_{\alpha\beta}}{\pd X_{\gamma}} M_{\gamma\beta} \dd t
	\notag\\
	&= \frac{1}{2} \insum_{\beta,\gamma}
		\left( \delta_{\alpha\beta\gamma} - \delta_{\alpha\gamma} x_{\beta} - \delta_{\beta\gamma} x_{\alpha} \right) \sigma_{\beta}
		\cdot
		X_{\gamma} \left( \delta_{\gamma\beta} - X_{\beta} \right) \sigma_{\beta} \dd t
	\notag\\
	&= \frac{1}{2} \insum_{\beta,\gamma}
	\left[
	\delta_{\alpha\beta\gamma}(1 - 2X_{\beta}) + \delta_{\alpha\gamma} X_{\beta}^{2} - \delta_{\beta\gamma} X_{\alpha} + \delta_{\beta\gamma} X_{\alpha} X_{\beta}
	\right] X_{\gamma} \sigma_{\beta}^{2} \dd t
	\notag\\
	&= \frac{1}{2} X_{\alpha}
	\left[
	(1 - 2X_{\beta})\sigma_{\beta}^{2} - \insum_{\beta} X_{\beta} (1 - 2X_{\beta}) \sigma_{\beta}^{2}
	\right] \dd t.
\end{flalign}

By substituting this correction back to \eqref{eq:SRD-Ito}, we see that the replicator dynamics with Stratonovich shocks \eqref{eq:SRD-Strat} are equivalent to the (Itô) stochastic replicator dynamics of exponential learning \eqref{eq:SXRD}.
In this context, \cite{MM10} showed that the conclusions of Theorems \ref{thm:dominated} and \ref{thm:folk} apply directly to the original, unmodified game $\game$ under \eqref{eq:SXRD}, so dominated strategies become extinct and strict equilibria are stochastically asymptotically stable under \eqref{eq:SRD-Strat} as well.
Alternatively, this can also be seen directly from the correction term \eqref{eq:correction} which cancels with that of \eqref{eq:pay-mod}.

\subsection{Random strategy switches}

An alternative source of noise to the players' evolution under \eqref{eq:RD} could come from random masses of players that switch strategies without following an underlying deterministic drift \textendash\ as opposed to jumps with a well-defined direction induced by a revision protocol.
To model this kind of ``mutations'', we posit that the relative mass $dX_{\alpha\beta}$ of players switching from $\alpha$ to $\beta$ over an infinitesimal time interval $dt$ is governed by the \ac{SDE}:
\begin{equation}
\label{eq:jumps}
dX_{\alpha\beta}
	= X_{\alpha}\left(\rho_{\alpha\beta} \dd t + dM_{\alpha\beta} \right),
\end{equation}
where $dM_{\alpha\beta}$ denotes the (conditional) mutation rate from $\alpha$ to $\beta$.

To account for randomness, we will assume that $M_{\alpha\beta}$ has unbounded variation over finite time intervals (contrary to the bounded variation drift term $X_{\alpha} \rho_{\alpha\beta} \dd t$).
Moreover, for concreteness, we will focus on the imitative regime where $\rho_{\alpha\beta} = x_{\beta} r_{\alpha\beta}$ and the mutation processes $M_{\alpha\beta}$ follow a similar imitative pattern, namely $dM_{\alpha\beta} = X_{\beta} dR_{\alpha\beta}$.
The net change in the population of $\alpha$-strategists will then be
\begin{equation}
\insum_{\beta} X_{\beta} \dd M_{\beta\alpha}
	- X_{\alpha} \insum_{\beta} \dd M_{\alpha\beta}
	= X_{\alpha} \insum_{\beta} X_{\beta} \dd Q_{\beta\alpha},
\end{equation}
where $dQ_{\beta\alpha} = dR_{\beta\alpha} - dR_{\alpha\beta}$ describes the \emph{net influx} of $\beta$-strategists in strategy $\alpha$ per unit population mass.
Thus, assuming that the increments $dQ_{\beta\alpha}$ are zero-mean, we will model $Q$ as an Itô process of the form:
\begin{equation}
\label{eq:mutation}
dQ_{\alpha\beta}
	= \eta_{\alpha\beta}(X) \dd W_{\alpha\beta}
\end{equation}
where $W_{\alpha\beta}$ is an ordinary Wiener process and the diffusion coefficients $\eta_{\alpha\beta}\from\strat\to\R$ reflect the intensity of the mutation process.
In particular, the only assumptions that we need to make for $W$ and $\eta$ are that:
\begin{equation}
\label{eq:symmetry}
dW_{\alpha\beta}
	= -dW_{\beta\alpha}
	\quad
	\text{and}
	\quad
\eta_{\alpha\beta}
	= \eta_{\beta\alpha}
	\quad
	\text{for all $\alpha,\beta\in\act$ and for all $k\in\play$,}
\end{equation}
so that the net influx from $\alpha$ to $\beta$ is minus the net influx from $\beta$ to $\alpha$;
except for this ``conservation of mass'' requirement, we will assume that the processes $dW_{\alpha\beta}$ are otherwise independent. 
Thus, in the special case of the ``imitation of success'' revision protocol \eqref{eq:success}, we obtain the \emph{replicator dynamics with random mutations}:
\begin{equation}
\label{eq:RMRD}
dX_{\alpha}
	= X_{\alpha} \left[ \payv_{\alpha}(X) - \insum_{\beta} X_{\beta} \payv_{\beta}(X) \right] dt
	+ X_{\alpha} \insum_{\beta\neq\alpha} X_{\beta} \eta_{\beta\alpha}(X) \dd W_{\beta\alpha}.
\end{equation}

This equation differs from \eqref{eq:SRD} in that the martingale term of \eqref{eq:SRD} cannot be recovered from that of \eqref{eq:RMRD} without violating the symmetry conditions \eqref{eq:symmetry} that guarantee that there is no net transfer of mass across any pair of strategies $\alpha,\beta\in\act$.
Nonetheless, by repeating the same analysis as in the case of Theorems \ref{thm:dominated} and \ref{thm:folk}, we obtain the following proposition for the stochastic dynamics \eqref{eq:RMRD}:

\begin{proposition}
\label{prop:mutations}
Let $X(t)$ be an interior solution orbit of the stochastic dynamics \eqref{eq:RMRD} and consider the noise-adjusted game $\mutgame$ with modified payoff functions:
\begin{equation}
\label{eq:pay-mod-mutations}
\mutpayv_{\alpha}(x)
	= \payv_{\alpha}(x)
	- \frac{1}{2} \insum_{\beta\neq\alpha} x_{\beta}^{2} \eta_{\beta\alpha}^{2}(x).
\end{equation} 
We then have:
\begin{enumerate}
[\indent\upshape(1)]
\item
If $p\in\strat$ is dominated in $\mutgame$, then it becomes extinct under \eqref{eq:RMRD}.
\item
If $\prob\left(\lim_{t\to\infty} X(t) = \eq\right) > 0$ for some $\eq\in\strat$, then $\eq$ is a Nash equilibrium of $\mutgame$.
\item
If $\eq\in\strat$ is stochastically Lyapunov stable, then it is a Nash equilibrium of $\mutgame$.
\item
If $\eq\in\strat$ is a strict Nash equilibrium of $\mutgame$, then it is stochastically asymptotically stable under \eqref{eq:RMRD}.
\end{enumerate}
\end{proposition}

\begin{proof}
The proof is similar to that of Theorems \ref{thm:dominated} (for Part 1) and \ref{thm:folk} (for Parts 2\textendash 4), so we omit it.
\end{proof}


\appendix
\section{Auxiliary results from stochastic analysis}
\label{app:Brownian}

In this appendix, we state and prove two auxiliary results from stochastic analysis that were used throughout the paper.
Lemma \ref{lem:Wbound} is an asymptotic growth bound for Wiener processes relying on the law of the iterated logarithm, while Lemma \ref{lem:hitprob} is a calculation of the probability that a Wiener process starting at the origin hits the line $a+bt$ in finite time.
Both lemmas appear in a similar context in \cite{BM14};
we provide a proof here only for completeness and ease of reference.

\begin{lemma}
\label{lem:Wbound}
Let $W(t) = (W_{1}(t),\dotsc,W_{n}(t))$, $t\geq0$, be an $n$-dimensional Wiener processes and let $Z(t)$ be a bounded, continuous process in $\R^{n}$.
Then:
\begin{equation}
\label{eq:Wbound}
f(t) + \int_{0}^{t} Z(s) \cdot dW(s)
	\sim f(t)
	\quad
	\text{as $t\to\infty$ \textup(a.s.\textup),}
\end{equation}
for any function $f\from[0,\infty)\to\R$ such that $\lim_{t\to\infty} \left(t\log\log t\right)^{-1/2} f(t) = +\infty$.
\end{lemma}

\begin{proof}
Let $\xi(t) = \int_{0}^{t} Z(s) \cdot dW(s) = \sum_{i=1}^{n} \int_{0}^{t} Z_{i}(s) \dd W_{i}(s)$.
Then, the quadratic variation $\rho = [\xi,\xi]$ of $\xi$ satisfies:
\begin{equation}
\label{eq:covest1}
d[\xi,\xi]
	= d\xi \cdot d\xi
	= \sum_{i=1}^{n} Z_{i} Z_{j} \delta_{ij} \dd t
	\leq M \dd t,
\end{equation}
where $M = \sup_{t\geq0} \norm{Z(t)}^{2} < +\infty$ (recall that $Z(t)$ is bounded by assumption).
On the other hand, by the time-change theorem for martingales \citep[Corollary~8.5.4]{Oks07}, there exists a Wiener process $\wilde W(t)$ such that $\xi(t) = \wilde W(\rho(t))$, and hence:
\begin{equation}
\frac{f(t) + \xi(t)}{f(t)}
	= 1 + \frac{\wilde W(\rho(t))}{f(t)}.
\end{equation}

Obviously, if $\lim_{t\to\infty} \rho(t) \equiv \rho(\infty) < +\infty$, $\wilde W(\rho(\infty))$ is normally distributed so $\wilde W(\rho(t))/f(t) \to 0$ and there is nothing to show.
Otherwise, if $\lim_{t\to\infty} \rho(t) = +\infty$, the quadratic variation bound \eqref{eq:covest1} and the law of the iterated logarithm yield:
\begin{equation}
\frac{\big\vert\wilde W(\rho(t))\big\vert}{f(t)}
	\leq \frac{\big\vert\wilde W(\rho(t)) \big\vert}{\sqrt{2\rho(t) \log \log \rho(t)}}
	\times \frac{\sqrt{2Mt \log \log Mt}}{f(t)}
	\to 0
	\quad
	\text{as $t\to\infty$,}
\end{equation}
and our claim follows.
\end{proof}

\begin{lemma}
\label{lem:hitprob}
Let $W(t)$ be a standard one-dimensional Wiener process and consider the hitting time $\tau_{a,b} = \inf\{t>0: W(t) = a + bt\}$, $a,b\in\R$.
Then:
\begin{equation}
\label{eq:hitprob}
\prob\left(\tau_{a,b} < \infty\right)
	= \exp(-ab-\abs{ab}).
\end{equation}
\end{lemma}

\begin{proof}
Let $\overline W(t) = W(t) - bt$ so that $\tau_{a,b} = \inf\{t>0: \overline W(t) = a\}$.
By Girsanov's theorem \citep[see e.g.][Chap.~8]{Oks07}, there exists a probability measure $\Q$ such that
\begin{inparaenum}
[\itshape a\upshape)]
\item
$\overline W$ is a Brownian motion with respect to $\Q$;
and
\item
the Radon\textendash Nikodym derivative of $\Q$ with respect to $\mathbb P$ satisfies
\begin{equation}
\left.\frac{d\Q}{d\mathbb P}\right\vert_{\filter_{t}}
	= \exp\left(-b^{2}t/2 + b W(t) \right)
	= \exp\left(b^{2}t/2 - b\overline W(t)\right),
\end{equation}
where $\filter_{t}$ denotes the natural filtration of $W(t)$.
\end{inparaenum}
We then get
\begin{flalign}
\prob\left(\tau_{a,b} < t\right)
	&= \ex_{\mathbb P}\left[ \one(\tau_{a,b}<t) \right]
	\notag\\
	&= \ex_{\Q} \left[ \one(\tau_{a,b}<t) \cdot \exp(-b^2t/2  -b \overline W(t)) \right]
	\notag\\
	&= \ex_{\Q} \left[ \one(\tau_{a,b}<t) \cdot \exp(-b^2\tau_{a,b}/2  -b \overline W(\tau_{a,b})) \right]
	\notag\\
	&= \exp(-ab)\mathbb E_{\mathbb Q} \left[ \one(\tau_{a,b}<t) \cdot \exp(-b^2\tau_{a,b}/2) \right],
\end{flalign}
and hence:
\begin{flalign}
\prob\left(\tau_{a,b} < \infty\right)
	&= \lim_{t \to \infty } \prob \left(\tau_{a,b}<t\right)
	\notag\\
	&= \lim_{t \to\infty }\exp(-ab) \ex_{\Q} \left[ \one(\tau_{a,b}<t) \cdot \exp(-b^2\tau_{a,b}/2) \right]
	\notag\\
	&= \exp(-ab) \ex_{\Q} \left[\exp(-b^2\tau_{a,b}/2) \right]
	\notag\\
	&= \exp(-ab-\abs{ab}),
\end{flalign}
where, in the last step, we used the expression $\ex[\exp(-\lambda \tau_{a})] = \exp(-a\sqrt{2\lambda})$ for the Laplace transform of the Brownian hitting time $\tau_{a} = \inf\{t>0:W(t) = a\}$ \citep{KS98}.
\end{proof}


\bibliographystyle{ametsoc}
\bibliography{IEEEabrv,ShockDynamics}

\end{document}